\documentclass[11pt]{amsart}

\usepackage[marginratio=1:1,height=600pt,width=374pt,tmargin=110pt]{geometry}

\usepackage{amsthm}
\usepackage{breqn}
\usepackage{amsmath,amsfonts,amssymb}
\usepackage{tikz}
\usepackage{hyperref}
\usepackage{microtype}
\usepackage{mathtools}
\usepackage{graphicx}
\usepackage{stmaryrd}
\usepackage{caption}

\usepackage[capitalize]{cleveref}             

\newcommand{\R}{\mathbb{R}} 
\newcommand{\e}{\varepsilon}
\newcommand{\N}{\mathbb N}

\newcommand{\C}{\mathbb C}
\DeclareMathOperator{\Ind}{ind}
\DeclareMathOperator{\id}{id}

\DeclareMathOperator{\SO}{SO}

\newcommand{\dist}{\operatorname{dist}}
\newcommand{\sgn}{\operatorname{sgn}}
\newcommand{\WM}{W^{1,2}(M)}
\newcommand{\irchi}[2]{\raisebox{\depth}{$#1\chi$}} 
\DeclareRobustCommand{\Chi}{{\mathpalette\irchi\relax}}

\newtheorem{theorem}{Theorem}[section]
\newtheorem*{theorem*}{Theorem}
\newtheorem{lemma}[theorem]{Lemma}
\newtheorem{corollary}[theorem]{Corollary}
\newtheorem{proposition}[theorem]{Proposition}

\newtheorem{claim}{Claim}
\newtheorem*{claim*}{Claim}

\theoremstyle{definition}

\newtheorem{definition}[theorem]{Definition}
\newtheorem*{remark*}{Remark}
\newtheorem{remark}[theorem]{Remark}

\title[Ground states of semilinear elliptic equations]{Ground states of semilinear elliptic equations}
\author[R.~Caju, P.~Gaspar, M.~A.~M.~Guaraco, H.~Matthiesen]{Rayssa Caju, Pedro Gaspar, Marco A.~M.~Guaraco and Henrik Matthiesen }
\address{The University of Chicago}
\email{guaraco@math.uchicago.edu, pgaspar@uchicago.edu, hmatthiesen@uchicago.edu, rayssacaju@gmail.com}

\thanks{The second author was was partially supported by Prof. Andr\'e Neves' Simons Investigator Award.}

\begin{document}

\begin{abstract}
We study solutions of $\Delta u - F'(u)=0$, where the potential $F$ can have an arbitrary number of wells at arbitrary heights, including bottomless wells with subcritical decay. In our setting, \textit{ground state} solutions correspond to unstable solutions of least energy. We show that in convex domains of $\R^N$ and manifolds with $\operatorname{Ric}\geq 0$, ground states are always of mountain-pass type and have Morse index 1. In addition, we prove symmetry of the ground states if the domain is either an Euclidean ball or the entire sphere $S^{N}$. For the Allen-Cahn equation $\e^2\Delta u - W'(u)=0$ on $S^{N}$, we prove the ground state is unique up to rotations and corresponds to the equator as a minimal hypersurface. We also study bifurcation at the energy level of the ground state as $\e\to 0$,  showing that the first $N+1$ min-max Allen-Cahn widths of $S^{N}$ are ground states, and we prove a gap theorem for the corresponding $(N+2)$-th min-max solution.
\end{abstract}

\maketitle

\section{Introduction}

We study solutions $u:\Omega \to \R$ to the semilinear elliptic equation
\begin{equation}\label{eq1}
 \begin{cases}\Delta u - f(u)=0 & \quad \mbox{in} \quad \Omega \\
 \partial_\nu u=0 & \quad \mbox{on} \quad \partial\Omega\end{cases}
\end{equation}
where $\partial_\nu$ is the normal derivative along $\partial\Omega$ and $f=F'$ is the derivative of a Morse function $F$, having finitely many non-degenerate critical points and ends no decaying faster than sub-critically. Some of the configurations that $F$ can have under these  assumptions are illustrated in Figure 1 below. 

\

Our hypothesis on $\Omega$ include strictly convex domains of $\R^N$ with smooth boundary, as well as closed manifolds with non-negative Ricci curvature which is not identically zero. More precisely, we assume:\\
\textit{
\begin{enumerate}
\item [(D)] $\Omega$ is one of the following types of domain:\\
\begin{enumerate}
\item [(D1) ]$\partial \Omega\neq \emptyset$ is smooth and strictly convex with $\operatorname{Ric}\geq 0$ on $\Omega$, or\\
\item [(D2) ]$\partial \Omega=\emptyset$ and $\operatorname{Ric}\geq 0$, with $\operatorname{Ric}> 0$ at some point of $\Omega$.\\
\end{enumerate}
\end{enumerate}
}

 \begin{figure}
  \centering
    \includegraphics[width=0.8\textwidth]{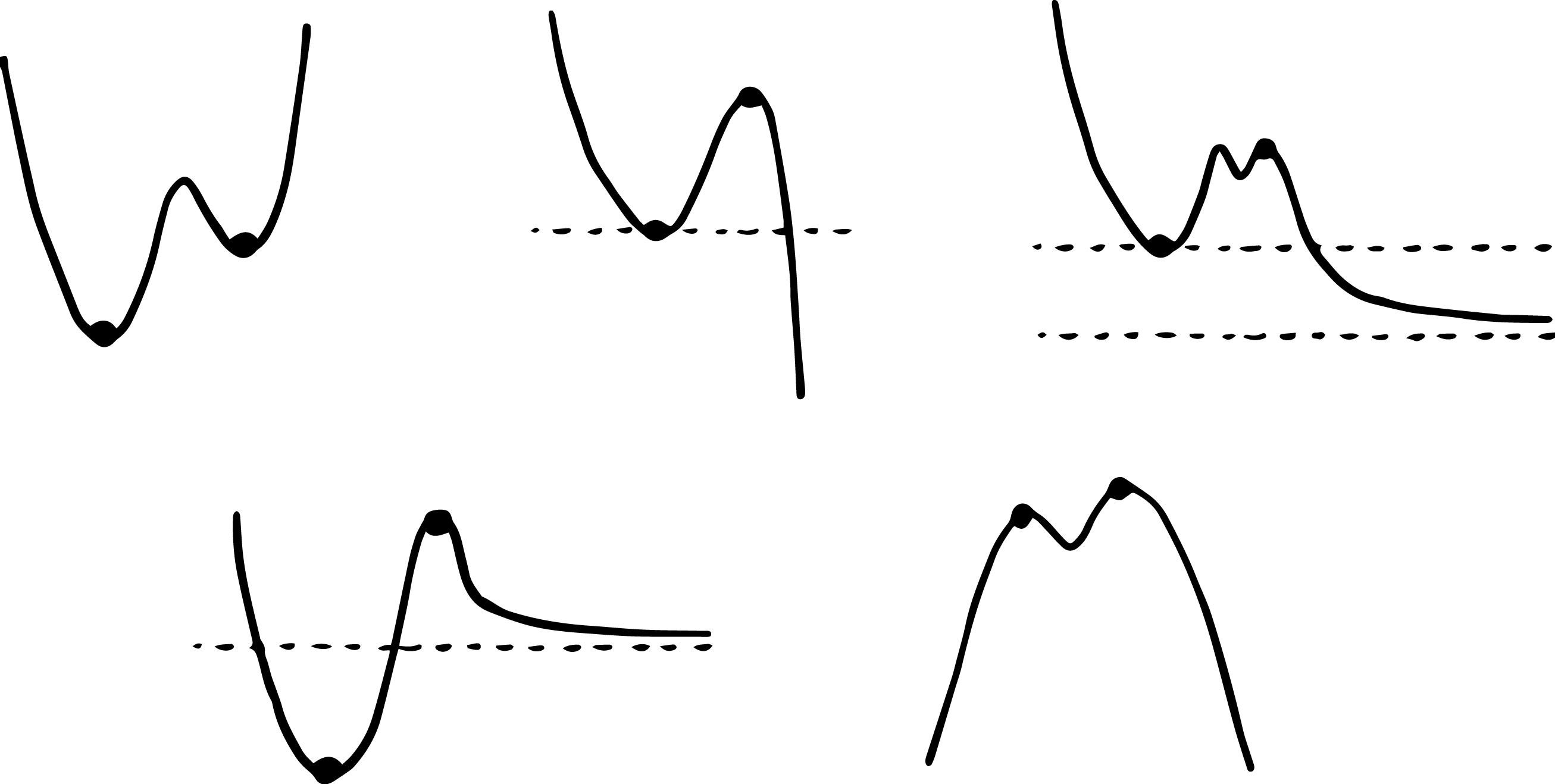}
      \caption{Some possible configurations of $F$}

\end{figure}

Solutions to equation \eqref{eq1} are critical points of the \textit{energy functional} \begin{equation}\label{functional}
E(u)=\int_\Omega \frac{|\nabla u|^2}{2}+F(u), \quad u \in W^{1,2}(\Omega).
\end{equation} The \textit{Morse index} of a solution $u$ is defined as the number of negative eigenvalues of $E''(u)(\cdot,\cdot)$, i.e.\ the number of negative modes of the linearization of \eqref{eq1} around $u$. When the Morse index of $u$ is zero we say that $u$ is \textit{stable}. Otherwise, we say that $u$ is \textit{unstable}.

\

The structure of stable solutions of \eqref{eq1} is well understood.
It is known that 
under hypothesis (D), stable solutions are the constants corresponding to local minima of $F$ (see \cite{EuclideanStable,Farina} and Theorem \ref{stableconst}).
 On the other hand, unstable solutions are abundant and their behavior is governed by the geometry of $\Omega$ (see \cite{PacardRitore, DelPino, CajuGaspar} for the case of the Allen-Cahn equation). 

\

In the present work, we pursue a general study of the simplest unstable solutions of equation \eqref{eq1}. Often called \textit{ground states}\footnote{We adopt this nomenclature for consistency with the literature on Partial Differential Equations, but we warn the reader that in Quantum Field Theory the therm \textit{ground state} is used to refer to the lowest critical energy level.}: these are solutions minimizing the energy among the set of all unstable solutions. Under the assumptions above, we characterize ground states as mountain pass critical points of Morse index 1. 

\

For some of the potentials we consider, this can not be obtained immediately by means of classical mountain pass methods (i.e.\ the existence of a mountain pass barrier and a Palais-Smale-type condition) since they fail to satisfy these requirements. To overcome this problem we prove a priori estimates for solutions of equation \eqref{eq1}. These estimates, which we believe are interesting on their own, imply that 
\begin{center}
\textit{potentials that do not decay faster than quadratically behave exactly like potentials with no decay at all}.
\end{center} In this way, we can study unstable solutions for a general potential $F$ with decay, as solutions for a modified potential $F^*$ better suited for the classical mountain pass theorem (see Figure 2 below).

\begin{figure}
  \centering
    \includegraphics[width=0.45\textwidth]{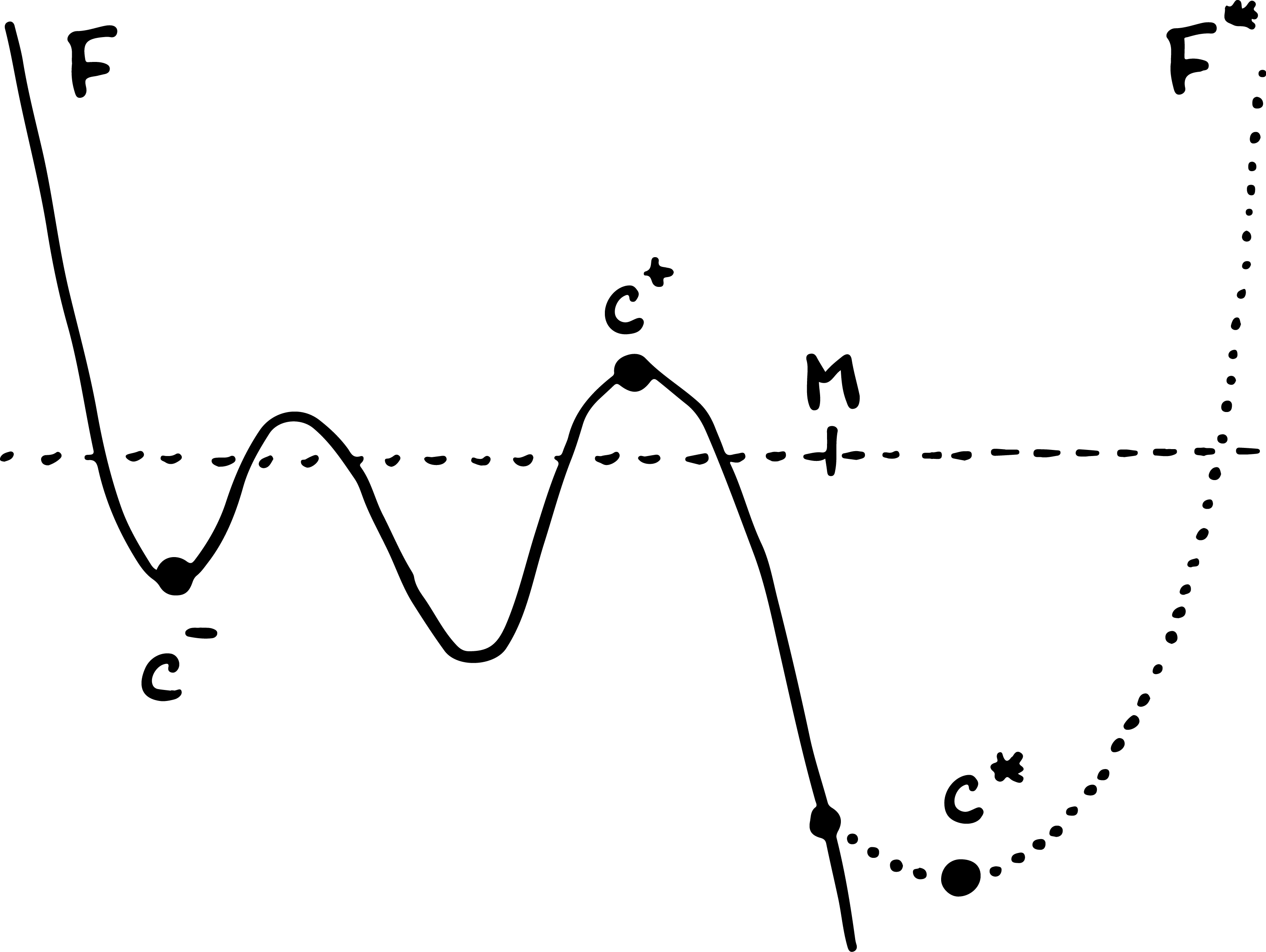}
      \caption{If $F$ does not decay faster than quadratically and $M$ is large enough, then  $F$ and $F^*$ induce the same unstable solutions}

\end{figure}

\

The mountain pass characterization of ground states allows us to derive some of their geometric properties. For example, we give simple conditions on $F$ that guarantee that ground states are non-constant, and we show that ground states are symmetric when the domain $\Omega$ is symmetric, e.g. a ball in $\R^N$ or the closed sphere $S^N$.
\\

 We then proceed to obtain improved symmetry results given additional structure of the nonlinearity.
 More specifically, we concentrate on the Allen-Cahn equation 
 $$
 \e^2 \Delta u- W'(u)=0
 $$ 
 on $S^N$, where $W$ is a symmetric double well potential. 
In this case, we can improve the symmetry results for ground states to uniqueness up to ambient isometries.
From this, we can draw a number of conclusions on the qualitative behavior of low energy solutions in parts in strong analogy with results for minimal hypersurfaces.

\

Firstly, in the case of the round sphere $S^N$ we can show that the ground state has nodal set precisely the equator, which is the least area minimal hypersurface.
Moreover, there is a gap of definite size to the next highest energy level achieved by a solution.
This, in particular, provides precise information on the first few Allen-Cahn widths, as defined in \cite{GasparGuaraco}.
 We show that the first $N+1$ widths are at the ground state level and there is a gap between these and the $(N+2)$-th width. 
 Under the analogy with minimal hypersurfaces, this is equivalent to saying that on $S^N$ the first $N+1$ widths of the area functional, as defined by Marques-Neves \cite{MarquesNevesRic}, correspond to the area of the equator. Following this analogy for $N=3$, one should expect that the 5-th width of $S^3$ is attained by a solution related to Clifford torus. We construct the candidate solution, but the problem of characterizing it as the one realizing the 5-th width remains open.
 This is related to the lack of an Urbano type theorem \cite{Urbano} for solutions of the Allen--Cahn equation.
 
 \

In a slightly different direction, we also obtain some control on the bifurcation behavior of solutions to the Allen--Cahn equation as $\e\to 0$. 
For the three-dimensional sphere, we determine the two largest values of $\e>0$ at which new solutions appear.
These are connected to low eigenvalues of the Laplacian.
While at the first bifurcation only ground states appear from constant solutions, at the second bifurcation, maybe surprisingly, there are at least two new solutions appearing.
One of them corresponds geometrically to the Clifford torus and the other to the intersection of $S^3$ with two orthogonal hyperplanes.

\

Both, our gap and bifurcation results depend in a quite subtle way on the symmetries of the ambient manifold.
For instance, we are at present not able to obtain the corresponding gap results, in fact not even the improved symmetry result, in the case of Euclidean balls.
The main difference to the case of the sphere is the the orbits of a point under the symmetry group of a ground state are much smaller than in the case of the sphere.
Similarly, our bifurcation analysis only applies to the three dimensional case at this point for related dimensional reasons.

\

The theory of ground states for semilinear elliptic equations in $\R^n$ has a long history and vast literature which we do not attempt to cover. The work of Berestycki and Lions  \cite{BerestyckiLions}, is one of the earliest and most influentials of the area. They study positive solutions in $\R^N$ assuming $F$ has a mountain-pass geometry in the sense that it has a non-degenerate minimum at $0$ with $F(0)=0$, is bounded from below by a subcritical polynomial, and there is a positive number $t>0$, where $F(t)=0$ (see hypothesis (1.1)-(1.3) in \cite{BerestyckiLions}). Some of their methods were based on the earlier work of Coleman, Glazer and Martin \cite{ColemanGlaserMartin}, which in turn was motivated by Coleman's study of vacuum instability in Minkowski space \cite{Coleman}. Despite the fact that ground states are not local minima, in both \cite{ColemanGlaserMartin} and \cite{BerestyckiLions} the scale invariance of $\R^N$ allows them to reduce the problem to a restricted minimization problem. In geometric terms, this is analogous to constructing a hypersurface with constant mean curvature  $H_0>0$, by first solving a problem of area minimization with a volume constrain, obtaining a sphere with curvature $H$ and then rescaling the sphere so that it has curvature $H_0$. In this way, it is possible to avoid the use of mountain pass arguments in the case of $\R^N$. Only until more recently the mountain pass characterization of the ground states was proved in this setting \cite{Jeanjean}. However, even in this case the scale invariance of $\R^N$ is used in the main arguments which cannot be reproduced in general domains.  

\

In recent years, there has been growing interest in proving similar results on domains that are not scale invariant, for example the Euclidean sphere $\Omega=S^N$. For Differential Geometry and Geometric PDE, the question is relevant under light of strong analogies between minimal surfaces and solutions to the stationary Allen-Cahn equation \cite{GasparGuaraco,GasparGuaraco2,ChodoshMantoulidis,Dey,Constante}. For Mathematical Physics, the question appears once more in the study of vacuum instability in curved space-times. For example, in the case of De Sitter space, critical points of a quantity often called \textit{Euclidean action} correspond to solutions of a semilinear elliptic equation over the Euclidean sphere $S^4$. These minimizers, known as \textit{bounce solutions}, are often assumed to have $\operatorname{O}(4)$ symmetry, a monotone profile and Morse index 1, but no proof of these facts has been given \cite{Thermal,Bounces}. The methods we develop are also motivated by the intention of filling this gap in the theory. These applications are discussed in an article under preparation by the third author and J. E. Camargo-Molina \cite{Eliel}.

\section{Main Results}

Let $F:\R\to\R $ be a Morse function with finitely many critical points, at least one of which is a local maximum. Denote by $c^-<c^+$ the smallest and largest critical points of $F$. Assume: \\
\textit{
\begin{enumerate}
\item [(A)] $F$ satisfies one of the following:\\
\begin{enumerate}
\item [(A1)]  $c^-$ and $c^+$ are local minima of $F$,\\
\item [(A2)] \textit{linear lower bound}: $\exists C>0$, such that $$-\sgn(t)f(t)\leq C(1+|t|), \ \ \forall t\in \R,$$
\item [(A3)] \textit{superlinear and subcritical decay}:
\begin{itemize} 
\item $c^-$ is a local minimum,
\item $\lim_{t\to+\infty} f(t)/t=-\infty$, 
\item $\exists C>0, R_0>0, \rho\in[0,1/2)$ and $p\in(2,\frac{2N}{N-2})$, such that $$-f(t)\leq C(1+|t|^{p-1}), \ \ \forall t>R_0$$ and $$-\rho tf(t)+F(t)-t^2\geq 0, \ \ \forall t>R_0.$$
\end{itemize}
\end{enumerate}
\end{enumerate}
}

\begin{remark*}
\

\begin{itemize}
\item A simple application of the maximum principle shows that for ground states to exist it is necessary for $F$ to have at least one local maximum. 
\item (A1) implies (A2). However, we separate them for convenience in stating the results below. In fact, all the configurations in Figures 1 are covered under (A2): $F(t)$ is allowed to grow freely as $|t|\to\infty$, but can decay at most quadratically. This includes potentials with an arbitrary number of wells, at arbitrary heights, with possibly some bottomless wells decaying no faster than quadratically.
\item (A3) only admits the configuration illustrated on the middle of the first row in Figure 1, with possibly additional critical points. It is satisfied when $F$ has the form $F(t)=-c|t|^{p}+G(t)$, for $t$ large enough, where $g(t)=G'(t)=o(t^{p-1})$. In this case, one can choose $\rho \in(1/p,1/2)$. The second inequality in (A3) is similar to the often called Ambrosetti-Rabinowitz condition.
\end{itemize}
\end{remark*}

\

Our first main result is the existence of ground states:

\begin{theorem}\label{thmexistence} Under hypothesis (A) and (D) equation \eqref{eq1} admits at least one ground state. 
\end{theorem}

After this, we present a useful mountain pass characterization of ground states that will later lead us to understand their geometric properties better. In order to state this characterization, we must introduce the following definition:

\begin{definition}\label{defoptimal}
Given $u\in W^{1,2}(\Omega)$, a continuous map $h:[-1,1]\to W^{1,2}(\Omega)$ joining $h(-1)$ and $h(+1)$ is said to be \textit{optimal at $u$ with respect to $E$} if
\begin{enumerate}
\item [a)] $u=h(0)$,
\item [b)] $E(u)=E(h(0))>E(h(t)),$ for all $0<|t|\leq1$
\item [c)] the map $E(h(t))$, $t\in[-1,1]$, is smooth near $t=0$, and 
\item [d)] $\frac{d^2}{dt^2}E(h(t))|_{t=0}<0$.
\end{enumerate}
\end{definition}

Our first variational characterization is the following, in particular generalizing
\cite[Theorem 2.1]{ACClosed}, which covers the case of the Allen--Cahn equation discussed in more detail below.

\begin{theorem}\label{thmgs1}
Assume (D) and either (A1) or (A3). If $u$ is a ground state of equation \eqref{eq1}, then
\begin{enumerate}
\item $u$ has Morse index 1 and is of mountain pass type with respect to $E$. More precisely, there exist constants $s^-<s^+$, such that $$E(u)=\inf_{h\in \Gamma} \sup_{t\in[-1,1]} E(h(t))>E(s^\pm)$$ where $\Gamma=\{h:[-1,1]\to W^{1,2}(\Omega) \ | \ h \text{ continuous, with } h(\pm1)=s^\pm \}.$ 
\item Moreover, there is a continuous path $h:[-1,1]\to W^{1,2}(\Omega)$ joining $s^-$ and $s^+$ which is optimal at $u$ with respect to $E$.
\end{enumerate}
\end{theorem}

In general, we have the following result:

\begin{theorem}\label{thmgs2}
Assume (D) and (A). Let $u$ be a ground state for equation \eqref{eq1}. Then, there exists $F^*:\R\to\R$ a Morse function with finitely many critical points, satisfying (A1) and such that (1) and (2) of Theorem \ref{thmgs1} holds with $E^*(\cdot)=\int_\Omega \frac{|\nabla \cdot|^2}{2}+F^*(\cdot)$ in place of $E$. \end{theorem}

\

Some of the steps in the proof of Theorem \ref{thmgs2} are interesting on their own. 
As mentioned in the introduction, we cannot guarantee a mountain pass geometry or a Palais-Smale condition for all the configurations in Figure 1. 
We overcome this by proving a priori estimates for unstable solutions of \eqref{eq1}. 
Informally, these estimates say that \textit{bottomless wells that decay no faster than quadratically behave exactly like finite wells}.

\

More precisely, let  $u^-$ and $u^+$ be the smallest unstable critical points of $F$. For simplicity defining the proposition below, denote  $k^-=\min(0,u^-)$ and $k^+=\max(0,u^+)$. In particular, all local maxima of $F$ lie on $[k^-,k^+]$, but there could be local minima of $F$ outside of this interval. We have the following a priori estimates:

\begin{proposition}\label{apriori}
Under hypothesis (A2) and (D), there exists a positive constant $M_0=M_0(\Omega,C,K)$ such that for any $u$ unstable solution to \eqref{eq1} it holds $$\|u\|_{L^\infty(\Omega)}+\|\nabla u\|_{L^\infty(\Omega)}\leq M_0,$$ where $K=\max\{|k^-|,k^+,\max_{[k^-,k^+]}|f|\}$ and $C$ is the constant from hypothesis (A1). 
\end{proposition}

Thanks to these estimates we can modify $F$ outside of a compact region without changing the set of unstable solutions of \eqref{eq1}, as long as we retain the same quadratic bound for the decay and we do not add new local maxima while doing so. We expect this estimate to be sharp, in the sense that it should not hold for potentials with superquadratic decay as a consequence of the classical work \cite{BahriLions}, neither it holds for stable solutions, since $F$ might have local minima outside of the interval $[-M_0,M_0]$. Proposition \ref{apriori} also implies that the set of unstable solutions of \eqref{eq1}. So one might expect that, generically, these type of equations only admit finitely many solutions (again, contrasting with the classical work \cite{BahriLions}). 

\

The mountain pass characterization from Theorems \ref{thmgs1} and \ref{thmgs2} also allows us to derive several geometric properties of ground states. For example, since the Morse index of a ground state must be one, in order to rule out constant ground states, it is enough for one of the smallest local maxima of $F$ to have Morse index at least 2. This is the content of the following corollary.

\begin{corollary}\label{nonconstant}
Assume (D) and (A). Let $c\in \R$ be a local maxima of $F$. Assume that $c$ minimizes $F$ in the set of all local maxima of $F$. If $F''(c)=f'(c)<-\lambda_1(\Omega)$, then $c$ has Morse index greater than 1 as a critical point of $E$. In this case, ground states of \eqref{eq1} are non-constant. 
\end{corollary}

We also prove that ground states are symmetric when the ambient space is symmetric. This is done by symmetrizing the optimal family obtained in the mountain pass characterization.
\begin{theorem} \label{thm_sphere}
Assume hypothesis (A), then:
\begin{enumerate}
\item if $\Omega=B^N_R$ is an Euclidean ball in $\R^{N}$, ground states of \eqref{eq1} are Foliated Schwarz symmetric. 
\item if $\Omega=S_R^{N}$ is an Euclidean sphere, ground states of \eqref{eq1} are Schwarz symmetric.
\end{enumerate}
\end{theorem}

We also study uniqueness of the ground state. Notice that in general, equation \eqref{eq1} might admit multiple ground states. However, it is still possible to show that ground states are unique in some important cases. 
We do this in particular for a symmetric double well potential on the sphere as discussed below.
Some of our results in fact apply to more general non-linearities as made precise in the corresponding sections.
\\ 

Consider the \emph{Allen-Cahn} equation
	\begin{equation}\label{AC} 
		\e^2\Delta u -W'(u)=0 \quad \mbox{in} \quad S^{N},
	\end{equation}
where $\e>0$ and $W(u)=(1-u^2)^2/4$. Solutions to this equation are the critical points of   $$E_\e(u)=\int_\Omega \e \frac{|\nabla u|^2}{2}+\frac{W(u)}{\e}.$$

\

We have following uniqueness result:

\begin{theorem} \label{mainac}
The ground state $u$ for equation \eqref{AC} on $S^{N}$ is unique up to rigid motions. The function $u$ is odd and Schwarz symmetric. In addition, the ground state is non-constant if and only if $\e \in (0,\e_1)$, where $\e_1=\sqrt{-W''(0)/\lambda_1}$ and $\lambda_1=\lambda_1(S^{N})$.
\end{theorem}

In other words, when $\e\in(\e_1,\infty)$, the ground state of the Allen-Cahn equation on $S^N$ is given by the constant solution $0$, and it is non-constant with nodal set exactly along an equator, when $\e\in(0, \e_1)$.

\

We also study bifurcation from the first energy level and the gap between the first and second energy levels of $E_\e$. More precisely, denote the first and second critical energy levels of $E$ as $a_\e\leq b_\e$, respectively, i.e.
\begin{equation} 
		a_\e = \inf\{ E_\e(u) : u \in \WM, E_\e'(u) = 0, E_\e(u)>E_\e(\pm1)\}
	\end{equation}
and
	\[b_\e = \inf\{ E_\e(u) : u \in \WM, E_\e'(u)=0, E_\e(u)>a_\e\}.\]
	
With this notation, we have the following gap theorem: 

\begin{theorem} \label{gap}
 $b_\e>a_\e$ if and only if $\e\in(0,\e_1)$. Moreover, the energy level $b_\e$ is realized by at least one solution which is non-constant provided $\e \in (0,\e_2)$, where $\e_2=\sqrt{-W''(0)/\lambda_2}$, where $\lambda_2=\lambda_2(S^{N})$.
\end{theorem}

\

Denoting  by $c_\e(1) \leq c_\e(2) \leq c_\e(3) \leq \dots $
the Allen--Cahn widths of the sphere as defined in \cite[Section 3.2]{ACClosed}, we also prove the following gap theorem:

\begin{theorem} \label{thm_width_gap}
For $\e>0$ sufficiently small we have on the sphere $S^N$ that
$$
c_\e(1) = \dots = c_\e(N+1) < c_\e (N+2).
$$
\end{theorem}

\

In virtue with the analogy with minimal surfaces \cite{GasparGuaraco,Dey}, we expect that for $N=3$, the critical level $c_\e(5)$ of Theorem \ref{thm_width_gap} is attained by solutions having their nodal set exactly on the Clifford torus of $S^3$. We construct this solution on \cref{clifford} below. Moreover, we prove:

\begin{theorem} \label{thm:bifurcation}
Let $\e_2 = (\lambda_2(S^3))^{-1/2}=\frac{1}{2\sqrt{2}}$. 

\begin{enumerate}
	\item[(1)] For any $\e \in (\e_2,\e_1)$, the only nonconstant solutions of the Allen-Cahn equation are the ground states (which are unique up to rotations). 
	\item[(2)] For $\e<\e_2$, there are at least two families of solutions which are not radially symmetric. These families of solutions have the symmetries of the Clifford torus and a pair of orthogonal equators, respectively, and accumulate on these minimal surfaces as $\e\downarrow 0$.
\end{enumerate}
\end{theorem}

\noindent\textbf{Organization.}  \cref{thmexistence} is proved in \cref{sec:existence}, \cref{thmgs1} and \cref{thmgs2} in \cref{sec:mpc}, \cref{apriori} in \cref{sec:apriori}, \cref{thm_sphere} in \cref{sec:sym}, \cref{mainac} in \cref{sec:sphere}, \cref{gap} and \cref{thm_width_gap} in \cref{sec:gap}, and \cref{thm:bifurcation} in \cref{sec:bif}.


\subsection*{Acknowledgements} We thank J. E. Camargo-Molina for pointing out the connections between semillinear elliptic equations and the study of vacuum stability. We also thank Manass\'es Xavier, Nestor Guillen, Manuel Del Pino and Andr\'e Neves for their interest in this work.

\subsection*{Notation.}\label{notation}

\

\medskip

\begin{tabular}{lll}
$B_R(p)$ && denotes the geodesic ball of radius $R$ centered at $p\in M$.\\
$|A|$ && denotes the $N$-dimensional Hausdorff measure of a subset\\
&&$A \subset M$, $|A|=\mathcal{H}^N(A)$.\\
$\Chi_A$ && denotes the \emph{characteristic function} of a set $A \subset M$, i.e. \\
&& $\Chi_A(x)=1$ if $x \in A$, and $\Chi_A(x) = 0$ otherwise.\\
$\|f\|_{p}$ && denotes the $L^p$ norm of $f$ on $M$, i.e. $\|f\|_p=\left(\int_M |f|^p\right)^{1/p}$.\\[2pt]
$c^-$ && denotes the smallest critical point of $F$.\\
$c^+$ && denotes the largest critical point of $F$.\\
$c_1<\cdots<c_n$ && is the list of \textit{unstable} critical point of $F$.\\
$k^-$ && denotes $\min(0,c_1)$.\\
$k^+$ && denotes $\max(0,c_n)$.\\
$\beta_{N}$ && denotes the volume of the $N$-dimensional sphere $S^{N}$, $\beta_N = |S^{N}|$.\\
\end{tabular}

\

\section{A priori estimates and compactness under (A2) and (D)}\label{sec:apriori}

Given $u$ a solution to \eqref{eq1}, let $P:\Omega\to \R$, be defined in terms of $u$ as \begin{align*}
P=\frac{|\nabla u|^2}{2}-F(u).
\end{align*}

The proofs of the following two lemmas follow closely those in Chapter 5 of \cite{Sperb}.  

\begin{lemma} 
Assume $\operatorname{Ric}\geq 0$. At points where $|\nabla u|\neq0$, the function $P$ satisfies a maximum principle. Namely, it holds \begin{align}\Delta P +  |\nabla u|^{-2}\langle \nabla P , \nabla P - \nabla|\nabla u|^2 \rangle \geq 0.\end{align}
\end{lemma}

\begin{proof}
From $\nabla P = |\nabla u|\nabla|\nabla u| - f(u)\nabla u $,  it follows that $$f(u)^2 = |\nabla u|^{-2} \langle \nabla P , \nabla P - \nabla|\nabla u|^2 \rangle + |\nabla|\nabla u||^2.$$

From Bochner's formula and $\Delta u =f(u)$, we obtain
\begin{align*}
\Delta P &= |\operatorname{Hess} u|^2 +\langle \nabla \Delta u, \nabla u \rangle + \operatorname{Ric}(\nabla u,\nabla u) - f'(u)|\nabla u|^2 -f(u)\Delta u \\
&= |\operatorname{Hess} u|^2 -f(u)^2+ \operatorname{Ric}(\nabla u,\nabla u).
\end{align*}
We obtain the inequality by combining both expressions, using $\operatorname{Ric}(\nabla u,\nabla u) \geq 0$ and the standard inequality $|\operatorname{Hess} u|^2 - |\nabla|\nabla u||^2\geq 0$. 
\end{proof}

\begin{lemma}
Let $\partial \Omega$ be strictly convex. Assume $u$ is a solution of \eqref{eq1} with either zero Dirichlet or zero Neumann boundary condition. Then, at a boundary point with $|\nabla u|\neq 0$, we must have $\partial_\nu P<0$.
\end{lemma}

\begin{proof}
For the Dirichlet case, notice that from the expression
$$f(u)=\Delta u= \langle \nabla_\nu \nabla u, \nu \rangle + \Delta_{\partial \Omega} u -\langle \vec{H}, \nabla u \rangle,$$ where $\vec{H}$ is the mean curvature vector of the level sets of $u$, and since $\nu$ and $\nabla u$ are parallel, we obtain $$f(u)=\bigg\langle \nu ,\frac{\nabla u}{|\nabla u|}\bigg\rangle\bigg(|\nabla u|^{-1}  \langle \nabla_\nu \nabla u, \nabla u \rangle- \langle\vec{H} ,\nu \rangle|\nabla u|\bigg).$$ Multipliying by $\langle\nu,\frac{\nabla u}{|\nabla u|} \rangle$ and rearranging the terms we get the desired inequality as long as $\partial \Omega$ is strictly mean-convex (this holds, in particular, when $\partial \Omega$ is strictly convex) $$0>\langle\vec{H} ,\nu \rangle|\nabla u|^2=  \partial_\nu\bigg( \frac{|\nabla u|^2}{2} - F(u)\bigg)=\partial_\nu P.$$

In the Neumann case, $\nu$ and $\nabla u$ are perpendicular along the boundary. We conclude 
\begin{align*}
\partial_\nu P & = \bigg\langle \nu, \nabla \frac{|\nabla u|^2}{2} - f(u)\nabla u \bigg \rangle\\
&= \langle \nu , \nabla_{\nabla u}\nabla u \rangle \\
&= - \langle \nabla_{\nabla u}\nu ,\nabla u\rangle \\
& = II_\nu(\nabla u,\nabla u)<0,
\end{align*}
where $II_\nu$ is the second fundamental form of $\partial \Omega$ with respect to the exterior normal $\nu$, which is a negative definite quadratic form when $\partial \Omega$ is strictly convex. 
\end{proof}

\begin{corollary}
Let $\Omega$ be bounded and strictly convex with $\operatorname{Ric}\geq 0$. Assume $u$ is a solution of \eqref{eq1} with either zero Dirichlet or zero Neumann boundary condition, then the maximum of $P$ in $\overline{\Omega}$ is attained at a critical point of $u$.
\end{corollary}

\begin{lemma}\label{unstableprop}
Let $c_1,\dots,c_n$ be the unstable critical points of $F$. If $u$ is an unstable solution to \eqref{eq1} then
\begin{enumerate}
\item $[u_\min,u_\max]\cap \{c_1,\dots,c_n\}\neq \emptyset$ and
\item $\max_{[u_\min,u_\max]} F \leq \max \{F(c_1),\dots F(c_n) \}$.
\end{enumerate}
\end{lemma}

\begin{proof}
Item (1) follows by contradiction directly from the maximum principle. In fact, assuming $[u_\min,u_\max]\cap \{c_1,\dots,c_n\}=\emptyset$ implies there is at most one local minimum of $F$ in $[u_\min,u_\max]$. If there are no local minimums, then $f'$ has a sign on this interval, contradicting $\Delta u=f'(u)$ either at the maximum or minimum of $u$ (depending on the sign of $f'$). On the other hand, if there is exactly one local minimum $m\in[u_\min,u_\max]$, then the maximum principle is not contradicted only if $u=m$. However, in this case the solution is stable by Lemma \ref{stable2}. 

Item (2) follows along similar lines. If there are no critical points outside of $[c_1,c_n]$ then $F$ decreases with respect to its distance to $[c_1,c_n]$ and the result follows. If there are critical points outside of $[c_1,c_n]$ then these are all local minima and there is at most one on each side of $[c_1,c_n]$. As above, by the maximum principle, these critical points bound $u_\min$ and $u_\max$ accordingly. Then, in this case $F$ decreases with respect to its distance to $[c_1,c_n]$ in $[u_\max,u_\min]$.   \end{proof}

\begin{proposition}\label{gradientbound} 
Assume (A2). There exists a positive constant $B_0=B_0(\Omega,C,K)$, such that for any $u$ which is an unstable solution to \eqref{eq1} we have $$\sup_{\Omega} |\nabla u| \leq B_0,$$ where $K=\max\{|k^-|,k^+,\max_{[k^-,k^+]}|f|\}$ and $C$ is given by hypothesis (A2).
\end{proposition}

\begin{proof}
Let $p_0 \in \overline{\Omega}$ be a point where $P$ attains its maximum and denote $u_0=u(p_0)$. By a previous result (above) we have that $\nabla u(p_0)=0$, so $P(p_0)=-F(u_0)$. 

\

\noindent\underline{\textbf{Case $u_0\in [k^-,k^+]$:}} We have the following inequalities on all of $\Omega$, 
\begin{align*}
|\nabla u|^2 & \leq F(u)-F(u_0) \\
&\leq F(c_k) - F(u_0) \\
&\leq |c_k-u_0| \max_{[k^-,k^+]}|f|\\
&\leq |k^+-k^-| \max_{[k^-,k^+]}|f|,
\end{align*} where the first inequality follows from $P\leq P(p_0)$, the constant $c_k \in \{c_1,\dots,c_n\}$ is such that $F(c_k)=\max\{F(c_1),\dots,F(c_n)\}$ and the second inequality  follows from Lemma \ref{unstableprop}.

\

\noindent\underline{\textbf{Case $u_0> k^+$:}} We subdivide the proof of this case into a series of claims. 

\begin{claim}
$u_\max =u_0$.
\end{claim}

\begin{proof}[Proof of Claim 1]Let $p^+ \in \Omega$ such that $u(p^+)=u_\max$. Clearly, $u_0\leq u_\max$ and $\nabla u(p^+)=0$. To proceed by contradiction assume $u_0<u_\max$. Since $F$ is strictly decreasing in $[k^+,u_\max]$, we have $F(u_\max)< F(u_0)$. On the other hand, by our assumption on $p_0$, we have $-F(u_\max)=P(p^+)\leq P(p_0)=-F(u_0),$ which is contradiction. This proves the claim. 
\end{proof}

\begin{claim}
For all $\delta \in (0,1)$, the gradient bound $|\nabla (u_0-u)| \leq A(\delta) |u_0-u|$ holds on the level set $\{ k^+ \leq u\leq u_0-\delta(u_0-k^+)\}$, with $A(\delta)=\sqrt{(2-\delta)C+ 2C\frac{(1+k^+ )}{\delta(u_0-k^+)} }.$
\end{claim}
\begin{proof}[Proof of Claim 2]
In what follows we use the fact that the function $u_0-u$ is positive on the level set $\{ k^+ \leq u\leq u_0-\delta(u_0-k^+)\}$, which follows from the previous claim. In the steps below there are three inequalities, the first one follows from $P\leq P(p_0)$, the second one uses the linear bound on $f$ and the fact that $u_0 \geq u\geq k^+\geq 0$, and the third one uses $u\in [k^+,u_0-\delta(u_0-k^+)].$

\begin{align*}
|\nabla (u_0-u)|^2 & = |\nabla u|^2\\
&\leq 2F(u)-2F(u_0) \\
& = 2\int_{u}^{u_0} -f(s) ds \\
& \leq 2C\int_{u}^{u_0} (1+s) ds \\
& = 2C(u_0-u) + C(u_0^2-u^2) \\
& = C\bigg( \frac{2 + (u_0+u)}{u_0-u}  \bigg) |u_0-u|^2\\
& \leq C\bigg( \frac{2 + (2u_0-\delta(u_0-k^+))}{\delta(u_0-k^+)}  \bigg) |u_0-u|^2\\
& = C\bigg(  (2-\delta)+ 2\frac{(1+k^+ )}{\delta(u_0-k^+)}  \bigg) |u_0-u|^2,
\end{align*}
which is what we wanted to prove.
\end{proof} 

\

\begin{claim} There exists $\delta\in(0,1)$, depending only on $\Omega$ and $C$, such that for any unstable solution we have
$$u_0 \leq k^+ + \frac{2d^2C(1+k^+)}{\delta\log(\delta)^2 +\delta(\delta-2)d^2C}.$$ 
\end{claim}

\begin{proof}[Proof of Claim 3]
Since $u$ is an unstable solution, by Proposition \ref{unstableprop} the set $\{c_1\leq u \leq c_n\}$ must be non-empty. In particular, since $u_\max=u_0$ by Claim 1, the level sets $\{u=s\}$, with $s\in[k^+,u_0]$ are non-empty.

Choose points $q_0\in \{u=u_0-\delta(u_0-k^+)\}$ and $q_1\in \{u=k^+\}$ and let $\gamma : [0,1]\to \Omega$ be a minimizing geodesic joining $q_0$ with $q_1$ (which exists because $\partial \Omega$ is strictly convex). We can always assume that $\gamma$ is contained in the set $\{k^+\leq u \leq u_0-\delta(u_0-k^+)\}$ by redefining $q_0$ to be $\gamma(t_0)$, where $t_0\in [0,1)$ is the last time $u(\gamma(t_0))=u_0-\delta(u_0-k^+)$. and $q_1$ to be $\gamma(t_1)$, where $t_1\in(t_0,1]$ is the first time $u(\gamma(t_1))=k^+$. If necessary, we can reparametrize this segment of geodesic so that it is defined again over $[0,1]$ and $|\gamma'(t)|=\operatorname{dist}(q_0,q_1)\leq \operatorname{diam}(\Omega)=d$, for all $t\in[0,1]$.

The gradient inequality obtained in Claim 2 then holds along $\gamma:[0,1]\to \{k^+\leq u \leq u_0-\delta(u_0-k^+)\}$ and denoting $(u_0-u)(t)=u_0-u(\gamma(t))$ it follows
$$(u_0-u)'(t)\leq |u'(t)|\leq |\gamma'(t)||\nabla u(\gamma(t))| \leq A(\delta) (u_0-u)(t).$$

From Gronwall's inequality we obtain $$u_0-k^+\leq \delta (u_0-k^+) \exp ( d \times A(\delta)),$$
which translates into $1 \leq \delta \exp( d \times A(\delta) )$, and then into $-\log(\delta) \leq d A(\delta)$, after taking logarithm on both sides. Squaring the expression, substituting the explicit formula for $A(\delta)$ and doing some simple arithmetic, one gets
$$(u_0-k^+)(\log(\delta)^2 +\delta d^2C -2 d^2C) \leq  \frac{2d^2C(1+k^+)}{\delta}.$$

Finally, choosing $\delta \in(0,1)$ small enough, we can guarantee that the quantity $\log(\delta)^2 +\delta d^2C -2 d^2C$ is positive. This proves the claim.
\end{proof}

\begin{claim}
Let $R_0=k^+ + \frac{2d^2C(1+k^+)}{\delta\log(\delta)^2 +\delta(\delta-2)d^2C}$ be the constant from the previous claim. Then $$|\nabla u|^2 \leq  |k^+-k^-|\max_{[k^-,k^+]}|f| + C(R_0^2-R_0).$$
\end{claim}

\begin{proof} We have the following inequalities
\begin{align*}
|\nabla u|^2 & \leq F(u)-F(u_0)  \\
&\leq F(c_k) - F(u_0)\\
&= \int_{c_k}^{k^+}-f(s)ds + \int_{k^+}^{u_0}-f(s)ds\\
&\leq |k^+-k^-|\max_{[k^-,k^+]}|f| + C\int_{0}^{R_0}(1+s)ds\\
&\leq |k^+-k^-|\max_{[k^-,k^+]}|f| + CR_0(1+R_0), 
\end{align*} where the first inequality follows from $P\leq P(p_0)$, the constant $c_k \in \{c_1,\dots,c_n\}$ is such that $F(c_k)=\max\{F(c_1),\dots,F(c_n)\}$ and the second inequality  follows from Proposition \ref{unstableprop}.
\end{proof}

Finally, notice that Claim 4 concludes the proposition when $u_0>k^+$ since $R_0$ depends only on $\Omega$, $C$ and $k^+$.

\

\noindent\underline{\textbf{Case $u_0< k^-$:}} This case can be handled using exactly the same computations as above but substituting the roles of $u_\max$, $u_0-u$ and $k^+$ in Claims 1-4, for $u_\min$, $u-u_0$ and $k^-$, respectively.

\end{proof}

\

\noindent\textbf{Proof of Proposition \ref{apriori}.}

\begin{proof} Let $B_0=B_0(\Omega,C,K_0)$ be the constant from Proposition \ref{gradientbound}, i.e.  $|\nabla u|\leq B_0$ at all points of $\Omega$. Since by Proposition \ref{unstableprop} the level set $\{c_1\leq u\leq c_n\}$ is non-empty, $\overline{\Omega}$ is compact and $\{c_1,\dots,c_n\}\subset [k^-,k^+]$, we obtain bounds on $\|u\|_{L^\infty(\Omega)}$ in terms of $\Omega$, $K$ and $B_0$.
\end{proof}

\

\begin{corollary}\label{a2compact}
Under hypothesis (A2) and (D), the space of solutions of equation \eqref{eq1} is compact.
\end{corollary}

\begin{proof}
This is an immediate consequence of the pointwise bounds obtained above. 
\end{proof}



\section{Compactness under (A3)}

Hypothesis (A3) assumes subcritical growth and a version of the Ambrosetti-Rabinowitz condition adapted to the case of zero Neumann boundary condition. The method of proof is standard but we include it for convenience of the reader.

\begin{lemma} \label{lem_ps}
Under hypothesis (A3), $E$ satisfies the Palais-Smale condition. 
\end{lemma}

\begin{proof}
Let $\{u_n\}_n \in W^{1,2}(\Omega)$ be a Palais-Smale sequence, i.e. $\sup_n E(u_n)=M<+\infty$ and $\|E'(u_n)\|\to 0$. Then
\begin{align*}
&M + \rho \|E'(u_n)\|\cdot (1+ \|u_n\|^2_{W^{1,2}(\Omega)})\\
&\geq M + \rho \|E'(u_n)\|\cdot \|u_n\|_{W^{1,2}(\Omega)}\\
&\geq E(u_n) - \rho E'(u_n)(u_n) \\
&=\int_\Omega \bigg(\frac{1}{2}-\rho\bigg) |\nabla u_n|^2 + \int_\Omega -\rho u_n f(u_n) + F(u_n) \\
&= \bigg(\frac{1}{2}-\rho\bigg)\|u_n\|^2_{W^{1,2}(\Omega)}+ \int_\Omega -\rho u_n f(u_n) + F(u_n) - \bigg(\frac{1}{2}-\rho\bigg) u_n^2\\
&\geq\bigg(\frac{1}{2}-\rho\bigg)\|u_n\|^2_{W^{1,2}(\Omega)}+ \int_\Omega -\rho u_n f(u_n) + F(u_n) -  u_n^2 \\
&\geq\bigg(\frac{1}{2}-\rho\bigg)\|u_n\|^2_{W^{1,2}(\Omega)}+ \int_{\{u_n\leq R_0\}} -\rho u_n f(u_n) + F(u_n) -  u_n^2.
\end{align*}
Since the second term of the last line is bounded and $\|E'(u_n)\|\to 0$, it follows that the sequence $u_n$ is bounded in $W^{1,2}(\Omega)$.

Using the Rellich-Kondrachov's compactness theorem, we can pass to a subsequence which is convergent in $L^q$, for $q=2$ and $q=p$, and weakly convergent in $W^{1,2}(\Omega)$, to a limit function $u$. Finally, we have that the last line in 
\begin{align*}
\int_{\Omega} |\nabla (u_n-u)|^2 &=\int_{\Omega} \nabla u_n\nabla (u_n-u) - \nabla u\nabla u_n+|\nabla u|^2 \\
&= 2E'(u_n)(u_n-u)-2\int_\Omega f(u_n)(u_n-u)- \nabla u\nabla u_n+|\nabla u|^2
\end{align*} 
 goes to zero. For the first term, we use that $E'(u_n)$ goes to zero and $u_n-u$ is bounded in $W^{1,2}(\Omega)$. For the second term, we can use Holder's inequality $$\bigg|\int_\Omega f(u_n)(u_n-u)\bigg|\leq \|f(u_n) \|_{L^{\frac{p}{p-1}}(\Omega)} \|u_n-u \|_{L^{p}(\Omega)}.$$  Since $|f(s)|$ is bounded by $C(|u|^{p-1}+1)$, the term $\|f(u_n) \|_{L^{\frac{p}{p-1}}(\Omega)}$ is bounded, while $\|u_n-u \|_{L^{p}(\Omega)}$ goes to zero. Finally, the strong convergence in $L^{2}(\Omega)$ and the weak convergence in $W^{1,2}(\Omega)$, imply that the last two terms cancel in the limit. \end{proof}

The following is an immediate consequence of the lemma above: 
 
\begin{corollary}\label{a3compact}
Under hypothesis (A3) and (D), the space of solutions to equation \eqref{eq1} with energy bounded from above, is compact.
\end{corollary}



\section{The mountain pass characterization}

In this section we provide the proofs of \cref{thmgs1} and \cref{thmgs2}.

\subsection{Preliminary results and some technical lemmas}\label{sec:pre}

In this section we collect definitions and results that will be useful in the forthcoming sections.  

\

\begin{theorem}[\cite{EuclideanStable,Farina}]\label{stableconst}
Under hypothesis $(D)$, stable solutions to equation \eqref{eq1} are constant functions.
\end{theorem}

\

\begin{lemma}\label{stable2}
Assume $\Omega$ satisfies $(D)$. If $F$ is a Morse function, then the following are equivalents:
\begin{enumerate}
\item [i)] $u$ is a non-degenerate local minimum of $E$.
\item [ii)] $u$ is a solution of \eqref{eq1} with Morse index 0 (i.e. stable).
\item [iii)] $u$ is constant equal to a local minimum of $F$.
\end{enumerate} 
\end{lemma}

\begin{proof}
i) $\implies$ ii) is basic calculus of variations. To see ii) $\implies$ iii), notice that by Theorem \ref{stableconst} $u$ must be a constant function. Then, $0=\Delta u=f(u)=F'(u)$, so $u$ is a critical point of $F$. Since $F$ is a Morse function, $u$ is either a local minimum or a local maximum. When $u$ is a local maximum then $F''(u)<0$, and the linearization of \eqref{eq1}, i.e. $-\Delta + F''(u)$, has $\phi\equiv1$ as an eigenfunction with negative eigenvalue $F''(u)$. Therefore, $u$ must be a local minimum, i.e. $F''(u)>0$. Finally, iii) $\implies$ i) follows from the  $E''(u)(v,v)=\int_{\Omega}|\nabla v|^2+F''(c)v^2\geq \int_{\Omega}|F''(c)v^2>0$, for all $v\neq 0$.
\end{proof}

\

The following is the classical Mountain Pass Theorem (see \cite{Ghoussoub}):

\begin{theorem}[Mountain pass]\label{mountainpass} Let $E:W^{1,2}(\Omega)\to \R$ be a $C^2$ energy functional given by $$E(u)=\int_\Omega\frac{|\nabla u|^2}{2}+F(u)$$ and $u^-,u^+\in W^{1,2}(\Omega)$ such that we have the following inequality  \begin{equation}\label{mpgeometry}\inf_{h\in \Gamma} \sup_{t\in [-1,1]} E(h(t))=E_0>E(u^\pm)\end{equation} where $\Gamma=\{h\in C([-1,1],W^{1,2}(\Omega)): h(\pm1)\equiv u^{\pm}\}.$

\

Assume there exists a sequence $\{h_n\}_{n\in\N}\subset \Gamma$, with $$E_0=\lim_{n\to\infty} \sup_{t\in[-1,1]} E(h_n(t))$$ and such that $E$ satisfies the Palais-Smale condition along $\{h_n\}_{n\in\N}$ (see Definition \ref{psalong} below). \\

Then, there exists $u \in W^{1,2}(\Omega)$ such that:
\begin{enumerate}
\item [a)] $u$ is a critical point of $E$,
\item [b)] $u=\lim_{n\to\infty}h_n(t_n)$, for some sequence $\{t_n\}_{n\in\N}\subset[-1,1]$.
\item [c)] $E(u)=E_0,$
\item [d)] $u$ has Morse index at most 1. Moreover, if $E$ does not admit degenerate critical points of Morse index 0, then $u$ must have Morse index 1. 
\end{enumerate}
\end{theorem}

\begin{definition}\label{psalong}
Under the hypothesis of the theorem above, we say that $E$ satisfies the Palais-Smale condition along $\{h_n\}_{n\in\N}$ if any sequence $\{u_n\}_{n\in\N} \subset W^{1,2}(\Omega)$ such that:
\begin{enumerate}
\item $u_n\in h_n([-1,1])$,
\item $\lim_{n\to\infty}E(u_n)=E_0$ and
\item $\lim_{n\to\infty}\|E'(u_n)\|=0,$
\end{enumerate} has a convergent subsequence.
\end{definition}

\begin{remark}\label{remarkmp}
 The following are standard situations in which Theorem \ref{mountainpass} can be applied:

\begin{enumerate}
\item If $u^-$ and $u^+$ are both strict local minima of $E$ in $W^{1,2}(\Omega)$, then inequality \eqref{mpgeometry} holds.
\item If $u^-$ is a strict local minima of $E$ in $W^{1,2}(\Omega)$ and there exists a ball $B(u^-,\delta)$, for some $\delta>0$, such that $\inf_{W^{1,2}(\Omega)\setminus B(u^-,\delta)} E\leq E(u^-)$, then there exists a $u^+ \in W^{1,2}(\Omega)\setminus B(u^-,\delta)$ for which inequality \eqref{mpgeometry} holds. 
\end{enumerate}
\end{remark}

\begin{lemma}[Deformation lemma]\label{morse}
Let $E$, $u^-$, $u^+$ and $E_0$ as in Theorem \ref{mountainpass}. Assume that $E$ does not admit degenerate critical points of index 0, and $h$ is a family joining $u^-$ with $u^+$ which is optimal at $u\in W^{1,2}(\Omega)$ with respect to $E$ (see Definition \ref{defoptimal}). If $E(u)=E_0$, then $u$ is a critical point of $E$ with Morse index equal to 1.  
\end{lemma}

\begin{proof}

To see that $u$ must be a critical point one uses the classical deformation argument from the mountain-pass theorem. It consists on pushing the path $\gamma$ in the direction of a $v \in W^{1,2}(\Omega)$ such that $E'(u)(v)<0$. We omit this part since it is standard, and the deformation argument we use later on this proof to show that the Morse index cannot be greater than 1, uses a similar construction.

We prove now that the Morse index of $u$ cannot be 0. Otherwise, $u$ would be a strict local minima of $E$ in $W^{1,2}(\Omega)$, since $E$ does not admit degenerate critical points of index 0. In particular, for $\delta>0$ small enough $E(h(\delta))>E(h(0))=E(u)$, which contradicts the assumption that $h$ is optimal at $u$.

To reach a contradiction, assume now that the Morse index of $u$ is greater or equal than 2. In this case, we can construct a competitor family on $\Gamma$ below the level $E_0$ in the following way. Let $v_1$ and $v_2$ the first two eigenfunctions of the linearization of equation \eqref{eq1}. By the Morse index assumption, we have $E''(u)(v,v)<0$, for all $v\in \operatorname{span}(v_1,v_2)$. Let $v=\alpha_1v_1+\alpha_2v_2$, where $\alpha_1,\alpha_2\in \R$, will be chosen later. Define $\gamma:\R \times [-1,1] \to W^{1,2}(\Omega)$ by \[\gamma(s,t):=h(t) + sv, \quad (s,t) \in \R \times [-1,1].\] To finish the lemma, we just have to show that $f(s,t)=E(\gamma(s,t))$ has a non-degenerate local maximum at $(s,t)=(0,0)$. If this is the case, then substituting a piece of $h$ by a path going around $(0,0)$ would do the work. 

To see that $(0,0)$ is a non-degenerate local maximum, notice that since $\gamma(0,0)=h(0)=u$, the point $(0,0)$ is a critical point of $E$ and, as a consequence, it is also a critical point of $f$. Therefore, it is enough to show that, for the right choices of $\alpha_1$ and $\alpha_2$, the Hessian of $f$ is negative definite at $(0,0)$. This is a simple computation:
\begin{align*}
\partial_{tt}^2 f(0,0)&=E''(u)(h'(0),h'(0))<0, \text{ since $h$ is optimal at $u$}.\\
\partial_{ss}^2 f(0,0)&=E''(u)(v,v)<0, \text{ since $v\in V$, and}\\ 
\partial_{st}^2 f(0,0)&= E''(u)(h'(0),v) \\
&=\int_{\Omega} \nabla h'(0) \nabla v + F''(u)h'(0)v\\
&=\int_{\Omega} h'(0)[-\Delta v + F''(u)v] + \int_{\partial \Omega}h'(0)\partial_\nu v\\
&=\int_{\Omega} h'(0)(\alpha_1\lambda_1 v_1 + \alpha_2\lambda_2 v_2) \\
&=0,
\end{align*}
where we are using that $\partial_\nu v=0$, since $v_1$ and $v_2$ are Neumann eigenvalues and we are selecting $\alpha_1$ and $\alpha_2$ so that the last term is equal to 0. 

\end{proof}

We also recall the following standard result for bounded semilinear parabolic heat flows (see \cite{Cazenave}):
\begin{theorem}[Parabolic flow]\label{flow}
Let $u^{\pm}_t : [0,+\infty) \to W^{1,2}(\Omega)$ be solutions to the parabolic equation \begin{align}\label{parabolic}\begin{cases}\partial_t u_t -\Delta u_t + f(u_t)=0  & \text{ on } [0,+\infty)\times\Omega \\ 
\partial_\nu u =0 & \text{ on } [0,+\infty)\times\partial\Omega. \end{cases}\end{align} Assume that $u_t^-<u_t^+$, for all $t\geq 0$ and $\sup_{t\geq 0}\|u^{\pm}\|_{L^\infty(\Omega)}<+\infty$. 

Then, given any $v_0 \in W^{1,2}(\Omega)$, with $u^-_0 < v_0 < u^+_0$, there exists a unique solution $v_t$ of \eqref{parabolic} with initial condition $v_0$ which exists for all $t\geq 0$ and such that: 
\begin{enumerate}
\item [a)] $u^-_t < v_t < u^+_t$, for all $t\geq 0$.
\item [b)] $E(v_t)$ is strictly decreasing on $t\geq 0$ (unless $v_t$ is a solution to the stationary equation $\eqref{eq1}$, in which case $v_t$ is constant). 
\item [c)] There exists $v_\infty \in W^{1,2}(\Omega)$ a solution to the stationary equation \eqref{eq1}, and sequence of times $t_k\to\infty$, such that $v_{t_k}\to v_\infty$, strongly in $W^{1,2}(\Omega)$.
\end{enumerate}
\end{theorem}

The following lemma will allow us to construct general optimal families for unstable solutions:

\begin{lemma}\label{optimalfam}
Let $u$ be an unstable solution to \eqref{eq1} and $\phi$ the first eigenfunction of the operator $-\Delta+F''(u)$. Assume 
\begin{enumerate}
\item $u<c^+$ (resp. $c^-<u$).
\end{enumerate}
Then, there exists $\delta\in(0,1)$ and a continuous map $h:[0,1]\to W^{1,2}(\Omega)$ (resp. $h:[-1,0]\to W^{1,2}(\Omega)$), such that 
\begin{enumerate}
\item $h(t)=u+t\phi$ for $t\in[0,\delta]$ (resp. $t\in[-\delta,0]$),
\item $\frac{d^2}{dt^2}E(h(t))|_{t=0}<0$,
\item $E(u)=E(h(0))>E(h(t)),$ for all $t\in (0,1]$ (resp. $t\in [-1,0)$),
\item $h(1)$ (resp. $h(-1)$) is a constant which is a stable critical point of $F$.
\end{enumerate}
\end{lemma}

\begin{proof}

We do the argument assuming $u<c^+$. The case $c^-<u$ is analogous. 

Given $\delta>0$, consider the smooth path $h:[0,\delta)\to W^{1,2}(\Omega)$, given by $h(t)=u+t\phi$. By our choice of $\phi$, and since $u$ is an unstable critical point of $E$, we have 
\begin{enumerate}
\item [i.] $\frac{d}{dt}|_{t=0}E(h(t))=0$, and\\
\item [ii.] $\frac{d^2}{dt^2}E(h(t))|_{t=0}=\int_\Omega\phi(-\Delta \phi+F''(u_0)\phi)=-\lambda\int_\Omega \phi^2<0$,
\end{enumerate}
where $\phi> 0$ everywhere on $\Omega$ and $\lambda>0$. Therefore, choosing $\delta\in(0,1)$ small enough, we can assume $E(u)=E(h(0))$ is a strict local maximum of $E(h(t))$ on $[0,\delta]$, and that $u<h(t)<c^+$, for all $t\in[0,\delta]$. 

It remains to continuously extend the path from $h(\delta)$ to a stable constant $t^+$, with $u<t^+\leq s^+$ without increasing the energy above $E(h(0))$.

Let $C^+=C^+(h(\delta))$ be the subset consisting of all the $v\in W^{1,2}(\Omega)$ such that 
\begin{itemize}
\item $v$ is a solutions to \eqref{eq1},
\item $u\leq v\leq  c^+$ and
\end{itemize} there exists a continuous map $h^+:[\delta,1]\to W^{1,2}(\Omega)$ satisfying: 
\begin{itemize}
\item $h^+(\delta)=h(\delta)$,
\item $E(h(\delta)) \geq E(h^+(t))$, for $t\in [\delta,1]$ and
\item $h^+(1)=v.$
\end{itemize} We want to show that $C^+(\delta)\neq \emptyset$ and that the minimum of $E$ in $C^+(h(\delta))$ is attained by a stable solution. 

We first prove that $C^+$ is not empty. Since $u$ and $c^+$ are stationary solutions of \eqref{parabolic} with $u<h(\delta)<c^+$, it follows from Theorem \ref{flow} that there is a unique $v\in C([0,\infty):W^{1,2}(\Omega))$ which solves the parabolic equation \eqref{parabolic} with initial condition $h(\delta)$. By item c. of Theorem \ref{flow}, for a large $T^+$, $v(T^+)$ is arbitrarily close in $W^{1,2}(\Omega)$ to a solution $v_\infty$ of \eqref{eq1}. Item (b) of Theorem \ref{flow}, implies $E(v_\infty) < E(h(\delta))$ and by continuity we can assume that $v(T^+)$ belongs to a convex neighborhood of $v_\infty$ in which $E< E(h(\delta))$. The path $v:[0,T^+]\to W^{1,2}(\Omega)$ can be completed to a path arriving at $v_\infty$, simply by joining $v(T^+)$ to $v_\infty$ with a straight line in $W^{1,2}(\Omega)$. Finally, we reparametrize the junction of both paths over the interval $[\delta,1]$. This shows $v_\infty$ is in $C^+$.

Since solutions to \eqref{eq1} are bounded between $u$ and $c^+$, classical Schauder estimates imply that, after passing to a subsequence, an energy minimizing sequence $\{v_n\}_{n\in\N}$ in $C^+$ must converge to some $v_\min \in W^{1,2}(\Omega)$ which is also a solution of \eqref{eq1}, with $E(v)< E(h(\delta))$. By arguing as before, for $n$ larger enough $v_n$ is in a convex neighborhood of $v_\min$ in which $E< E(h(\delta))$. Therefore, the path joining $h(\delta)$ with $v_n$ can be completed to a path arriving at $v_\min$ through a straight line. This shows that $v_\min$ also belongs to $C^+$. 

Finally, notice that $v_\min \in C^+$ must be stable. Otherwise, we could repeat the argument and find a continuous path joining $v_\min$ to a solution energy strictly less than that of $v_\min$. Connecting this path with the one joining $h(\delta)$ with $v_\min$, we would contradict that $v_\min$ attains the minimum of $E$ on the set $C^+$. Therefore, $v_\min$ must be stable and by Lemma \ref{stable2} it is a constant equal to a local minimum of $F$ and of $E$. 
\end{proof}

The proof of the following lemma is elementary and it is left to the reader.

\begin{lemma}\label{fstar}
Let $F:\R\to\R$ be a Morse function with finitely many critical points. Let $c^+$ be the largest critical points of $F$ and assume that $c^+$ is a local maximum. Then, for any given $M>c^+$, there exists a Morse function $F^*:\R\to\R$ with finitely many critical points, such that:
\begin{enumerate}
\item $F^*=F$ on $(-\infty,M]$,
\item $F^*\geq F$ on $(M,+\infty)$,
\item $F^*$ only has one critical point $c^*$ on the interval $(M,+\infty)$, which is a local minimum. 
\end{enumerate}
\end{lemma}

\begin{remark}\label{remarkfstar}
Similarly, if $c^-$ is a local maximum, we could find $F^*$ that coincides with $F$ on the $[c^-,c^+]$ and that has at most one critical point to each side of this interval, which is a local minima. 
\end{remark}

\subsection{Proof of Theorem \ref{thmexistence}}\label{sec:existence}

By assumption $F$ has at least one local maximum $c$. From $F'(c)=f(c)=0$ and $F''(c)<0$ and classical properties of the Laplacian operator, it is easy to check that the constant function $u\equiv c$ is a solution to \eqref{eq1} and that for any constant direction $a\neq 0$, we have $E''(c)(a,a)<0$. Therefore, $u$ is an unstable solution, i.e. the set of unstable solutions \eqref{eq1} is not empty. 

Take a sequence of unstable solutions $\{u_n\}_{n\in\N}$ which is minimizing for $E$.  From Corollary \ref{a2compact} and Corollary \ref{a3compact}, after perhaps passing to a subsequence, these must converge to some $u$ which is also a solution of \eqref{eq1}. We just need to argue that $u$ is also unstable. This is a consequence of the fact that the set of stable solutions is isolated, since these are strict local minimizers. In fact, by item i) of Lemma \ref{stable2} a stable solution must be a local minimum of $F$.

\subsection{Proofs of Theorem \ref{thmgs1} and Theorem \ref{thmgs2}}\label{sec:mpc}

Let $u$ a ground state of \eqref{eq1}. We divide the proofs of these theorems into several cases depending on the hypothesis assumed on $F$. In each case, we first construct an optimal family for $u$, and then we prove $u$ is the solution of a mountain pass problem. The cases are presented in such a way that information from earlier cases can be used in later cases.\\

\noindent\textbf{Case 1: $F$ satisfies (A1)}. In this case, $c^\pm$ are both local minima of $F$. By the maximum principle we must have $c^-<u<c^+$, so we can apply Proposition \ref{optimalfam} twice to conclude there exist $t^-<t^+$, which are local minima of $E$ and $F$, and a family $h$ joining $t^-$ with $t^+$ optimal at $u$ with respect to $E$. Moreover, $c^-\leq t^-< h(t) < t^+\leq c^+$, in particular $t^-< u=h(0) < t^+$. This gives us the optimal family for this case.

Now, we check the hypothesis to apply the mountain pass theorem. By Item (1) in Remark \ref{remarkmp}, the inequality in the hypothesis of Theorem \ref{mountainpass} holds for our choice of $t^-$ and $t^+$. Let $E_0$ and $\Gamma$ be as in the statement of Theorem \ref{mountainpass}. Given $\{h_n\}_{n\in\N}\subset \Gamma$ we define its truncation to $[c^-,c^+]$ as $\tilde h_n(t)=\min(c^+,\max(c^-,h_n(t)))$. Since $F$ is decreasing on $(-\infty,c^-]$ and increasing on $[c^+,+\infty)$, it follows that after truncation its energy can only decrease, i.e. $E(\tilde h_n(t))\leq E(h_n(t))$. In particular, there exists $\{\tilde h_n\}_{n\in\N}\subset \Gamma$, such that $E_0=\lim_{n\to\infty}\sup_{t\in[-1,1]}E(\tilde h_n(t))$ and $c^- < \tilde h_n(t)<c^+$. Since $\{\tilde h_n\}_{n\in\N}$ is a family of bounded functions with bounded energy and $F$ is bounded from below, it is easy to see that their Sobolev norm is bounded. Then, a simple application of the Rellich-Kondrachov Compactness Theorem, gives us that the Palais-Smale condition holds along $\{\tilde h_n\}_{n\in\N}$. We can apply Theorem \ref{mountainpass} and conclude that there exists a critical point $u_{0}$ of $E$, with $E(u_{0})=E_0$. In addition, by Lemma \ref{stable2}, $E$ does not admit degenerate stable critical points, therefore Theorem \ref{mountainpass} implies $u_{0}$ is unstable. 

Finally, we must check that $u$ is also a solution of the mountain pass problem described in the previous paragraph and has Morse index 1. In fact, $u$ belongs to a family $h \in \Gamma$, which implies $E_0 \leq E(u)$. On the other hand, $u$ is a ground state and since $u_0$ is unstable, we must have $E(u)\leq E(u_{0})=E_0$. We conclude that $E(u)=E_0$ which is the mountain pass critical level. Finally, the existence of the family $h$ optimal at $u$ with respect to $E$, and Lemma \ref{morse}, imply that $u$ has Morse index 1. 

This proves both Theorem \ref{thmgs1} and Theorem \ref{thmgs2} when $F$ satisfies (A1).\\

\noindent\textbf{Case 2: $F$ satisfies (A2)}. If both $c^-$ and $c^+$ are stable we are in Case 1 above. Therefore, we can assume that $c^-$ is stable and $c^+$ is unstable (the remaining cases are similar and we comment on them at the end of this proof). Let $M_0$ be the one from Proposition \ref{apriori}. Then, $u\leq M_0$ and $c^+\leq M_0$. Fix $M>M_0$, and let $F^*$ and $c^*$ be given by Lemma \ref{fstar} for this choice of $M$. Let $E^*(v)=\int_{\Omega}\frac{|\nabla v|^2}{2}+F^*(v)$. Since $F$ and $F^*$ coincide in the range of $u$, it follows that $u$ is also an unstable critical point of $E^*$. From the construction of $F^*$ and Proposition \ref{apriori}, a function $v$ is an unstable critical point of $F$ if and only if it is a critical point of $F^*$. It follows that $u$ is also ground state for $F^*$. Since, $F^*$ satisfies (A1), we are in the situation of Case 1, which impliea Theorem \ref{thmgs2} for $F$.

The first remaining case is: $c^-$ unstable and $c^+$ stable, which is symmetric to what we just did. Finally, there is the case: both $c^-$ and $c^+$ are local maxima. To deal with this, we can use the estimates Proposition \ref{apriori} to argue as in Lemma \ref{fstar} also the left of $c^-$, (e.g. by applying it to $\tilde F(x)=F(-x)$ and then using $\tilde F^*$). See Remark \ref{remarkfstar}.

\

\noindent\textbf{Case 3: $F$ satisfies (A3).} As before, we first construct a family which is optimal at $u$ with respect to $E$. By the maximum principle $c^-<u$. If $u_\max<c^+$, then we can apply the exact same construction of Case 1, to obtain an optimal family joining $t^-<t^+$, stable critical points of $F$, such that $c^-\leq t^-<u<t^+<c^+$. If instead $c^-<c^+\leq u_\max$. Let $t^-$ be the largest stable critical point of $F$ such that $t^-<u$. Choose $M$, such that $M>u_{\max}\geq c^+$ and $F(t^-)\geq F(M)$, which exists by assumption (A3). Let $F^*$ and $c^*$ be given by Lemma \ref{fstar} for this choice of $M$. Since $u<M$ and $F=F^*$ on $(-\infty,M)$, it follows that $u$ and $t^-$ are critical points of $E^*$ and $t^-<u<c^*$. As in Case 1, applying Proposition \ref{optimalfam} we obtain a family $h$ joining $t^-$ and $c^*$ which is optimal at $u$ with respect to $E^*$. By item 2 of Lemma \ref{fstar}, we have $E^*\geq E$, and $E^*(u)=E(u)$ it follows that $h$ is also optimal at $u$ with respect to $E$.

Now we prove that $u$ is a mountain pass critical point. By Lemma \ref{lem_ps}, $F$ satisfies the Palais-Smale condition and by Remark \ref{remarkmp} we have a mountain pass barrier in both cases considered in the previous paragraph. As in Case 1, applying Theorem \ref{mountainpass} we obtain an unstable critical point at the same energy level as $u$. Once more, using Lemma \ref{morse}, $u$ has Morse index 1, and must be a mountain pass solution for paths joining $t^-$ with $t^+$, in the first case, and paths joining $t^-$ with $c^*$, in the second case. This finishes the proof of Theorem \ref{thmgs1} when $F$ satisfies (A3).


\section{Symmetry of ground states}\label{sec:sym}

In this section we prove Theorem \ref{thm_sphere}. We first discuss the symmetry of ground states assuming that $\Omega$ is the unit sphere $S^N$ in $\R^{N+1}$, and in the last subsection, the axial symmetry for ground states in the unit ball $B_1^{N}$ in $\R^{N}$, following \cite{BartschWethWillem}. 

\

\subsection{Symmetrization in the sphere} Fix $z_{0}$ in $S^{N}$ and write $d(y) = \dist(y,z_{0})$ for the geodesic distance in $\Omega = S^N$. Our strategy is the following is to work with a path which is optimal at a ground state. The \emph{symmetrization} of this path does not increase its energy. On the other hand, the optimality of the path guarantees that the solution must coincide with its symmetrization, thus proving the result. In order to do this, a key issue that must be addressed is the continuity of the symmetrized path.

\

One of the most natural notions of symmetrization in $S^N$ is the \emph{symmetric decreasing rearrangement}, which provides a \emph{radially symmetric} and decreasing function $u^*:S^N \to \R$, i.e. it can be written as a decreasing function of $d(y)$, and such that $\{u>t\}$ and $\{u^*>t\}$ have the same measure. 

\

Given a Borel function $u: S^{N}\rightarrow \R$, the associated \emph{distribution function} $\mathcal{V}_{u}: \R\rightarrow [0,\beta_N]$ is defined by 
	\begin{equation}\label{vu}
	 \mathcal{V}_{u}(s) := \int_{S^{N}} \Chi_{\{u>s\}}\,d\mathcal{H}^{N} = |\{x\in S^{N}: u(x) > s\}|.
	\end{equation}
	Denote by $R_{u}(s)$ the unique nonnegative real number such that 
	\begin{equation*}\label{ru}
		|B_{R_{u}(s)}| = \mathcal{V}_{u}(s),
	\end{equation*} 
	i.e., such that any geodesic ball of radius $R_{u}(s)$ has volume $\mathcal{V}_{u}(s)$.	The \emph{symmetric decreasing rearrangement} (or simply the \emph{symmetrization}) of $u$, with respect to $z_{0} \in S^{N}$, is the function $u^{*}: S^{N} \rightarrow \R$, defined by
	\begin{equation}\label{sdr}
	u^{*}(y) := \int_{\R}\Chi_{B_{R_{u}(s)}}(y)\,ds,
	\end{equation}
	
	We remark that, inasmuch as
\begin{equation*}
\Chi_{B_{R_{u}(s)}}(y) = \Chi_{\left\{\,t\in \R : \, R_{u}(t) > d(y)\,\right\}}(s) \quad \mbox{for all} \quad y \in S^{N}, s \in \R,
\end{equation*}
and $R_{u}$ is nonincreasing, we may express $u^*$ explicitly as
\begin{equation*}
u^{*}(y) = 	\sup\left\{s: \mathcal{V}_{u}(s)>|B_{d(y)}(z_{0})|\,\right\}.
\end{equation*}

The main properties of the symmetrization are summarized below.

\begin{proposition} \label{sdrproperties}
Let $\Omega=S^N$ and let $u \in L^1(\Omega)$. 

\begin{enumerate}
	\item For any $y,y'\in \Omega$,
		\[d(y) \leq d(y') \implies u^*(y)\geq u^*(y').\]
	In particular, $u^*$ is radially symmetric.
	\item For any Borel function $\Phi: \R \to \R$, if $\Phi(u) \in L^1(\Omega)$, then $\Phi(u^*) \in L^1(\Omega)$, and it holds
		\[\int_{\Omega}\Phi(u^*) = \int_{\Omega}\Phi(u).\]
	\item If $u \in W^{1,p}(\Omega)$ for some $1 \leq p \leq \infty$, then $u^* \in W^{1,p}(\Omega)$, and it holds
		\[ \|\nabla u^*\|_{p} \leq \|\nabla u\|_{p}.\]
	\item Assume that $u \in W^{1,p}(\Omega)$ for some $1<p<\infty$ and that the set $\{x \in \Omega : \nabla u(x) = 0 \}$ has zero $N$-dimensional Hausdorff measure. If
		\[\|\nabla u^*\|_p = \|\nabla u\|_p,\]
	then, there is an isometry $T$ of $S^N$ such that $u \circ T=u^*$.
\end{enumerate}
\end{proposition}

\begin{proof}The first property follows directly from the definition of $u^{*}$. The proof of (2) can be found in \cite{BaernsteinII}, see Proposition 1.18. Finally, the proof of (3) and (4) is due to Brothers and Ziemer \cite{BrothersZiemer}. 
\end{proof}

As a consequence, we see that
	\begin{equation}
	E(u^*) = \int_\Omega \frac{|\nabla u^*|^2}{2} + F(u^*) \leq \int_\Omega \frac{|\nabla u|^2}{2} + F(u) = E(u)
	\end{equation}
 for any $u \in W^{1,2}(\Omega)$. Moreover, $E(u^*)=E(u)$ if and only if $\|\nabla u\|_2 = \|\nabla u^*\|_2$. If $u$ is a non-constant solution to \eqref{eq1}, then $\{\nabla u = 0\}$ has zero $\mathcal{H}^N$ measure (see \cite{GarofaloLin}), so from Item (4) above, we get $u=u^*$, provided this equality holds.
 
 \

Consider the optimal path $h$ given by Theorem \ref{thmgs2}, and let $h^*:[-1,1] \to W^{1,2}(S^N)$ be its symmetrization given by $h^*(t) = (h(t))^*$. Even though $E(h^*(t)) \leq E(h(t))$ and $h^*( \pm 1) = h(\pm 1)$, it is not straightforward that $h^*$ is continuous, as $u \mapsto u^*$ is \emph{not} a continuous on $W^{1,2}(S^N)$. In fact, the continuity problem for the symmetrization was studied in $\R^N$ by Almgren and Lieb  \cite{AlmgrenLieb}. On one hand, they proved that the symmetrization is continuous as a map in $W^{1,2}$, exactly at functions that satisfy a condition called \emph{co-area regularity}, which is met for any $C^{N-1,1}_{\mathrm{loc}}$ function in $\R^N$ (see \cite[Theorem 5.2]{AlmgrenLieb}). On the other hand, they constructed a dense set of functions, all in $C^{N-1,\lambda}_{\mathrm{loc}}(\R^N)$, which are \emph{not} co-area regular.


We believe it is possible to extend the results of Almgren and Lieb to the symmetrization in $S^N$. Since the optimal path $h$ in Theorem \ref{thmgs2} is constructed using eigenfunctions of the stability operator and a parabolic flow, the function $h(t)$ is smooth (hence co-area regular) for every $t$, so it is natural to expect that the symmetrized path $h^*$ is continuous. However, we have opted for a more economical approach, using a simpler notion of rearrangement called \emph{polarization}. The relevant properties of polarization are that it approximates the symmetrization and at the same time is continuous in $W^{1,2}$. The use of polarizations to prove symmetry of solutions of partial differential equations has appeared in a number of works (see \cite{VanSchaftingenWillem, Enea1, BSPTW, Pacella}).

\subsection{Polarization} Remember we are in the case $\Omega=S^N$. Denote by $\mathcal{H}$ the family of closed halfspaces $H$ of $\R^{N+1}$, such that $0 \in \partial H$. For $H\in \mathcal{H}$ consider the reflection $\sigma_{H}: \R^{N+1}\rightarrow \R^{N+1}$ with respect to the hyperplane $\partial H$.  By fixing $H\in \mathcal{H}$ we can define the polarization of a function $u: \Omega \to \R$ with respect to the hyperplane $\partial H$ as the function $u_H:\Omega \to \R$ given by
	\begin{equation} \label{def:polarization}
u_{H}(x) = \left\{
\begin{array}{lcc}
\max\{u(x), u(\sigma_{H}(x))\},\quad  x\in \Omega\cap H \\
\min\{u(x), u(\sigma_{H}(x))\},\quad x \in \Omega \backslash H
\end{array}
\right.
	\end{equation}

The polarization compares the values of $u$ on both sides of $\partial H$, and keeps the larger value in $H$. We will denote by $\mathcal{H}_* \subset \mathcal{H}$ the set of closed halfspaces $H$ for which $z_0 \in H$. The following proposition gathers some useful properties of the polarization which will be used later in the paper. 

\begin{proposition} \label{prop:polar}
 Let $\Omega = S^N$ and let $H \in \mathcal{H}$.
	\begin{enumerate}
		\item For any Borel function $\Phi: \R \to \R$, if $\Phi(u) \in L^1(\Omega)$, then ${\Phi(u_H) \in L^1(\Omega)}$, and it holds
		\[\int_{\Omega}\Phi(u_H) = \int_{\Omega}\Phi(u).\]
		\item If $u \in W^{1,p}(\Omega)$ for some $1 \leq p \leq \infty$, then $u_H \in W^{1,p}(\Omega)$, and it holds
		\[ \|\nabla u_H\|_{p} = \|\nabla u\|_{p}.\]
	\end{enumerate}
\end{proposition}

These properties can be easily proved noting that each $\sigma_H$ is an isometry, and that $u_H$ can be written explicitly in terms of $u$ and $u\circ \sigma_H$, see e.g. \cite[Sections 3.3 and Chapter 7]{BaernsteinII}. \\

As discussed above, another key feature of the polarization is

\begin{proposition} \label{prop:continuous}
	For any closed halfspace $H \in \mathcal{H}$, the map $u \in W^{1,p}(S^{N}) \mapsto u_H \in W^{1,p}(S^{N})$ is continuous.
\end{proposition}

We now discuss the connection between symmetrization and polarization which will be used in the proof of Theorem \ref{thm_sphere}. First, we note that
\[ u = u_* \iff u=u_H, \quad \mbox{for all} \quad H \in \mathcal{H}_*.\]
In fact, if $H \in \mathcal{H}_*$ and $u=u_H$, then $u \geq u \circ \sigma_H$ in $\Omega \cap H$, and $u \leq u \circ \sigma_H$ in $\Omega \setminus H$. It this holds true for any $H \in \mathcal{H}_*$, then one readily checks that $u$ is radially symmetric and decreasing in the radial direction, so that $u=u^*$

Finally, we will use the fact that one may approximate the symmetric decreasing rearrangement $u^*$ by a sequence of polarizations strongly in $L^2$ norm.

\begin{theorem}[\cite{VanSchaf}] \label{approx-polar}
	Let $\Omega = S^N$. There exists a sequence $\{H_k\}$ in $\mathcal{H}_*$ such that, for any $u \in W^{1,2}(\Omega)$, if $\{u_k\} \subset W^{1,2}(\Omega)$ is defined by $u_0=u$ and
	\[u_{k+1} = (u_k)_{H_1,\ldots,H_{k+1}},\]
	then $\|u_k - u^*\|_{L^2(\Omega)} \to 0$ and $u_k \rightharpoonup u$ in $W^{1,2}(\Omega)$. 
\end{theorem}

\begin{proof}[Proof of Theorem \ref{thm_sphere} (2)] Without loss of generality, we may assume $S^N_R = S^N$. Let $u$ be an unstable solution of least energy. By Theorem \ref{thmgs2}, there exists $h \in \Gamma$ such that $h(0) = u$, and 
	\[E(h(t)) < E(u) = E_1 = \adjustlimits \inf_{\gamma \in \Gamma} \sup_{s \in [-1,1]} E(\gamma(s)), \ \mbox{for all} \ t \in [-1,1]\setminus\{0\}.\]
Given a halfspace $H\in \mathcal{H}_*$, let
	\begin{equation*}
	h_{H}(t): = (h(t))_{H},
	\end{equation*}
be the polarization of the path $h$ with respect to $H$. By Propositions \ref{prop:polar} and \ref{prop:continuous}, we have $h_{H} \in \Gamma$ and $E(\gamma_{H}(s)) = E(\gamma(s))$, for all $s \in[-1,1]$. Consequently,
	\begin{equation}\label{blau}
	E(u_{H}) = \sup_{s\in [-1,1]} E(h_{H}(s)) = \sup_{s\in [-1,1]} E(h(s)) = E_1
	\end{equation}
We claim that $h_{H}(0) = u_{H}$ is also a solution of \eqref{AC}. Indeed, if that is not the case, we can construct a new path $\bar{h} \in \Gamma$ by perturbing the original path in a neighborhood of $u_{H}$ in the direction of a function $\phi$ such that $\langle E'(u_{H}), \phi \rangle < 0$. Then, there exists a small $\delta >0$ such that $E(\bar{h}(t)) < E(h_{H}(t))$ for all $|t|<\delta$, and $E(\bar{h}(t)) \leq E(h_{H}(t))$, hence $\sup E(\bar h) < \sup E(h) = E_1$, which contradicts \eqref{blau}.
	
	By Theorem \ref{approx-polar}, there exists a sequence of closed halfspaces $\{H_{k}\}_{k}$ such that the sequence $\{u_k\}$ defined by $u_0 = u$ and $u_{k+1} := (u_{k})_{H_{1}\cdots H_{k+1}}$ converges to the symmetrization $u^{*}$ strongly in $L^{2}(S^{N})$. By the arguments above, we see that each $u_{k}$ is an unstable solution of \eqref{eq1}, with energy $E_1$ and constant $W^{1,2}$ norm. By the compactness of the solutions, after possibly passing to a subsequence, $u_{k}$ converges to a solution $\bar{u} \in W^{1,2}(S^{N})$ strongly in $W^{1,2}(S^{N})$. Therefore, $\bar{u} = u^{*}$ and 
	\begin{equation*}
	E(u^{*}) = \lim_{n\to\infty}E(u_{n}) = E(u).
	\end{equation*}
	Since $\int_{S^{N}} F(u^{*}) = \int_{S^{N}}F(u)$, for all $k$, we get $\int_{S^{N}}|\nabla u^{*}|^{2} = \int_{S^{N}}|\nabla u|^{2}$. Since $u$ is nonconstant, the set $\{x \in S^N : \nabla u(x) = 0\}$ has zero measure (see \cite{GarofaloLin}). By Proposition \ref{sdrproperties}, we see that $u$ and $u^{*}$ agree, up to ambient isometries, and hence $u$ is radially symmetric. \qedhere
\end{proof}

\begin{remark} \label{alternative-symmetry}
We believe that in the proof above $u=u_{H}$ should hold for any $H$, which would imply the rotational symmetry  and monotonicity of $u$. To achieve this, after proving that $u_H$ is a solution, one could try to follow the steps in \cite{BartschWethWillem}, where it is done for domains in $\R^N$ (see also \cite{VanSchaftingenWillem,Enea1}). In fact, in Subsection \ref{symball}, we use this approach to deal with the case of the Euclidean ball. These results rely on the strong maximum principle. The method we present above for the case of the sphere, is slightly different and relies on the unique continuation property and the rigidity of the \emph{P\'olya-Sz\"ego inequality}  (see Proposition \ref{sdrproperties}).
\end{remark}

\

\subsection{Symmetrization in the unit ball}\label{symball} Consider now the case $\Omega = B_1^{N}$. In this case, we do not expect that ground states of \eqref{eq1} have radial symmetry, so we need to introduce a different type of symmetrization.

Our geometric motivation is the connection between Neumann solutions to the Allen-Cahn equation \eqref{AC} and \emph{free boundary minimal hypersurfaces}. Solutions whose nodal sets accumulate on such minimal hypersurfaces (satisfying a nondegeneracy condition), were constructed by Pacard and Ritor\'e in \cite{PacardRitore}. Conversely, for families of solutions with bounded energy as $\e \downarrow 0$, the energy density accumulates on a (possibly singular) minimal hypersurface (see also \cite{MizunoTonegawa} for the case of Neumann solutions). Since least area free boundary minimal hypersurfaces in $B_1^N$ are flat equatorial disks, it seems reasonable to expect that the ground states in $B_1^N$ inherit this symmetry.

Consequently, they are expected to be \emph{foliated Schwarz symmetric}: this means that they are axially symmetric with respect to the axis generated by some $z_0 \in S^{N-1}$, and decreasing with respect to the polar angle from this axis. Similar symmetry results were proved in \cite{BartschWethWillem} and \cite{VanSchaftingenWillem}, for Dirichlet solutions, and in \cite{Enea1} for Neumann solutions and a sublinear potential with a unique critical point. 

As mentioned in Remark \ref{alternative-symmetry} above, once we characterize any ground state $u$ as mountain pass type solutions, one proves that its polarization $u_H$ is a solution for any half-space $H$. Using the argument of \cite{BartschWethWillem} (see Lemma 2.5 and Theorem 2.6), we see that $u=u_H$ for any half-space containing $z_0 \in S^{N-1}$ where $z_0 = \frac{x_0}{|x_0|}$ and $x_0 \in B_1^N\setminus \{0\}$ is such that
	\[u(x_0) = \max \{u(x) : x \in B_1^N, \ |x|=|x_0|\}.\]
This proves that $u$ is foliated Schwarz symmetric (with respect to the $z_0$), and finishes the proof of \ref{thm_sphere} (1).

\

\begin{remark}
We believe that proof of Theorem \ref{thm_sphere} for the sphere, can be adapted to the case of the Euclidean ball as well. The corresponding rearrangement notion is called \emph{cap symmetrization} \cite[\S 7.5]{BaernsteinII}, and it is defined as the symmetric decreasing rearrangement in each sphere in $B_1^N$ centered at the origin, see also \cite{Weth}. In this case, one still needs to show that the solution is \emph{nonradial} and derive a rigidity statement similar to item (4) in Proposition \ref{sdrproperties} for cap symmetrization under some condition on the nodal set of $\nabla u$ in each sphere. A related rigidity result was obtained for \emph{Steiner} symmetrization in \cite{CianchiFusco} (on the other hand, see \cite[Example 5.5]{Weth}).
\end{remark}

\begin{remark}
The results of \cite{WeiWinter} imply that for $f(t) = t-|t|^{p-1}t$, for $1<p<\frac{N+2}{N-2}$ (or $p>1$ for $N=1,2$), ground states are \emph{odd} with respect to the hyperplane orthogonal to the axis of symmetry. This implies that $\{u=0\}$ is precisely $\{x \in B_1^N : x_N=0\}$. We expect the same to be true for the Allen-Cahn equation. In the next section, we the analogous result in the case of $S^N$.
\end{remark}

\section{The case of the Allen-Cahn on \texorpdfstring{$S^N$}{S^N}}\label{sec:sphere}

We now turn to study the case of the Allen--Cahn equation
\begin{equation} \label{sec_eq_ac}
\e^2\Delta u - W'(u)=0
\end{equation}
with 
$W(u)=(1-u^2)^2/4$ in more detail.
In fact, some of our results apply to more general non-linearities.
In addition to our previous assumptions we assume from here on that the right hand side in \eqref{eq1} satisfies
\begin{itemize}
	\item[(i)] $f(t)=-f(-t)$ for any $t \in \mathbb{R}$.
	\item[(ii)] $-f(t)/t$ is non-increasing in $t \geq 0$.
\end{itemize}
Note that these hold in particular for the standard double well potential $W$.



Since $W$ satisfies (A1), we have that \cref{thmgs1} applies to \eqref{sec_eq_ac} whenever the ambient manifold satisfies (D).
Also observe that the set of solutions having positive energy is always nonempty, as $E_\e(0) = \frac{|M|}{4\e}$. Moreover, 
	\[E_\e(u) = \int_{M} \left(-\frac{\e}{2}u\Delta u + \frac{(1-u^2)^2}{4\e}\right) \, d\mathcal{H}^n = \frac{|M|}{4\e} - \frac{\int_{M} u^4}{\e} = E_\e(0) - \frac{\int_{M}u^4}{\e}.\]
for any solution $u$ of \eqref{AC}, so $E_\e(u)<E_\e(0)$, provided $u$ does not vanish identically.

From \cref{stableconst} any nonconstant solution of a semilinear elliptic PDE on a compact manifold with positive Ricci curvature is unstable.
 In particular, if a least energy solution is nonconstant, then \cref{thmgs1} implies that it is a min-max solution with Morse index 1, and there is an optimal path joining this solution to the absolute minimizers $\pm 1$ of $E_\e$. 

On the other hand, as noted in \cite{ACClosed}, for large $\e$, the only solutions of \eqref{AC} are the constant solutions. 
The proof is based on the following classical result, which holds for solutions to \eqref{eq1} under the assumptions on $f$ stated above.

\begin{theorem}[\cite{BrezisOswald}] \label{Dirichlet}
Let $\Omega \subset M$ be a domain with nonempty smooth boundary. The boundary value problem
	\[ \left\{ \begin{array}{rl} \e^2\Delta u- f(u) = 0 & \mbox{ in } \Omega\\ 
	u  > 0,& \mbox{ in } \Omega\\
	u = 0,& \mbox{ on } \partial \Omega.\\  \end{array}\right. \]
has at most one solution. Moreover a solution exists if and only if $\e<\lambda_1(\Omega)^{-1/2}$.
\end{theorem}

Consider now the case $M=S^{N}$, endowed with the round metric of constant curvature 1.
 It follows from the result above and the Faber-Krahn inequality that if \eqref{AC} has a nonconstant solution, then
	\[\e < \lambda_1(S^{N}_+)^{-1/2}= \frac{1}{\sqrt{N}},\]
where $S^{N}_+ = S^{N} \cap \{x \in \R^{n+1} : x_{N+1} >0\}$. Conversely, if this inequality holds, then we can find a solution of \eqref{AC} on $S^{N}$ whose nodal set is the equator $S^{n-1} \subset S^{N}$ by putting $u(x) = u_+(x)$, if $x \in S^{N}_+$, and $u(x) = -u_+(-x)$, otherwise, where $u_+(x)$ is the unique positive solution of the problem above on $\Omega=S^{N}_+$. Furthermore, the uniqueness part of Theorem \ref{Dirichlet} guarantees that $u_+$, and hence $u$, are radially symmetric.

We now collect some consequences of \cref{Dirichlet} and the maximum principle for rotationally symmetric solutions.
Before we start it is useful that assumption (ii) on $f$ above in particular implies that 
$$
f'(\theta t) \leq \theta f'(t)
$$
for $t \geq 0$ and $\theta \in (0,1)$.
This in turn implies that if $u$ is a non-negative solution to \eqref{eq1} then $\theta u$ is a subsolution.

\begin{lemma} \label{nodal-equator}
	Assume that $f$ is as above.
	Let $u \colon S^{N} \to \mathbb{R}$ be a non-constant solution to \eqref{eq1}, which is rotationally symmetric
	about the $e_{N+1}$-axis and has connected nodal set.
	Then $u$ is odd under the reflection at the hyperplane $\{x_{N+1}=0\}$.
\end{lemma}

\begin{proof}
	First note that the assertion follows from \cref{Dirichlet} once we know that $\{u=0\}=\{x_{N+1}=0\}$, since $u$ a solution if and only if 
	$-u$ is a solution thanks to the assumption on $f$.
	Moreover, note that \cref{thm_sphere} implies that the nodal set $\{u=0\}$ has to be connected, since $u$ is monotone.
	
	Let us assume now that $\{u=0\} \neq \{x_{N+1}=0\}$.
	By the rotational symmetry of $u$ it follows that $\{u=0\}$ is contained in the interior of a hemisphere.
	In particular, there is a connected component $\Omega$ of $S^{N} \setminus \{u = 0\}$ such that $\bar \Omega$ is contained
	in the interior of a hemisphere.
	We can thus find some rotation $R \in SO(N+1)$ such that
	$$
	\Omega \cap R(\Omega) = \emptyset \ \text{and} \ \partial \Omega \cap \partial (R (\Omega)) = \{z\}
	$$ 
	for some point $z \in S^{N}$.
	We write $\Omega' = R (\Omega)$.
	By our assumption on the potential, we then have functions $u,v \colon \Omega' \to \mathbb{R}$ both solving \eqref{eq1} and such that
	$u \geq 0$ in $\Omega'$, $u=0$ along $\partial \Omega'$, and $v > 0 $ in $\bar \Omega' \setminus \{z\}$, $v(z)=0$.
	(Here we use that $\bar \Omega$ is contained in the interior of a hemisphere.)
	But this can be seen to impossible using the maximum principle by an argument similar to that in \cite[Corollary 7.4]{gmn_ac_19}.
	
	Here are the details. 
	Since $v >0$ in $\bar \Omega' \setminus \{z\}$ and $\partial_\nu v (z) > 0$ for $\nu$ the outward pointing normal of $\Omega'$,
	we can choose $\theta < 1$ such that 
	\begin{equation} \label{eq_max_comp}
	\theta u < v  \ \text{in} \ \bar \Omega' \setminus \{z\}.
	\end{equation}
	As explained above, the assumptions on $f$ imply that $\theta u$ is a subsolution in $\Omega'$.
	Since $v$ is a solution, it follows from the maximum principle, that \eqref{eq_max_comp} continues to hold for $\theta = 1$.
	(Note that we can not have $\theta u = v$ for any $\theta$ by construction.)
	At this stage an application of the Hopf boundary lemma implies that 
	$$
	\partial_\nu u(z) < \partial_\nu v(z)
	$$
	which is impossible by construction.
\end{proof} 



\section{Integrability of the kernel and energy gap}\label{sec:gap}

In this section we prove \cref{gap}.
This follows from some fairly standard arguments once we have obtained the integrability of the kernel at rotationally symmetric, index $1$ solutions.
This is established in the following subsection.

\subsection{The kernel at a rotational symmetric solution}

The goal of this subsection is to show that any function in the kernel of the linearized operator at a rotationally symmetric 
solution $u$ with index $1$ and nodal set the equator, is generated by an ambient isometry, i.e.\ the kernel of the linearized operator is integrable.

We keep making the assumption of $f$ from the preceding section.
It is useful to note the second assumption implies that
 $$
 tf'(t)-f(t) \ \text{has a sign for} \ t>0.
 $$

Assume now that $u$ is a solution to \eqref{eq1} and denote by 
$$
L_u=-\varepsilon^2 \Delta + f'(u)
$$
the linearized operator at $u$.
If $X$ is a Killing vector field on $S^{N}$, then it is generated by a rotation with axis around some vector $v \in S^{N}$. 
The function $\phi_X=\langle X , \nabla u \rangle$ clearly lies in the kernel of $L_u$.
When $u$ is rotationally symmetric, say about the $e_{N+1}$-axis, and odd under the reflection at the hyperplane $\{x_{N+1}=0\}$, the nodal set 
$$
\{ \phi_X =0\} = S_X,
$$
where $S_X$ is an equator containing $v$ and $e_{N+1}$, provided $\phi_X \neq 0$.
The latter is equivalent to $v\neq \pm e_{N+1}$.

\begin{proposition} \label{prop_nullity}
Let $u$ be a non-constant, rotationally symmetric solution to \eqref{eq1} with index $1$ and nodal set the equator.
Then we have that
$$
\dim \ker L_u = N.
$$
\end{proposition}

\begin{proof}
As above, after a rotation, we may assume that $\{u=0\} = \{x_{N+1}=0\}$, i.e.\ $u$ is rotationally symmetric about the $e_{N+1}$-axis and odd with respect to the reflection at the hyperplane $\{x_{N+1}=0\}$.
For $v \in S^{N}$ denote by $r_v \colon S^{N} \to S^{N}$ the reflection at the hyperplane $v^\perp$, i.e.\
$$
r_v(x) = x - 2 \langle x , v \rangle v.
$$
For simplicity we write $r_{j}=r_{e_j}$ for the reflection at the hyperplanes $\{x_j=0\}$.
Note that the function $W''(u)$ is invariant under $r_v$ if $v \in S^{N-1} = S^{N} \cap \{x_{N+1} =0 \}$, since $u$ is rotationally symmetric about the $e_{N+1}$-axis, or if $v=e_{N+1}$, since $u$ is odd with respect to $r_{N+1}$ and $W''$ is even.
Since each $r_v$ is also an isometry of $S^{N}$, we find that each $r_v$ commutes with $L_u$, in particular it acts on $\ker L_u$.
Moreover, each $r_v$ is an involution.
In particular, we can decompose $\ker L_u$ into the $\pm1$ eigenspaces of $r_v$.

\begin{claim} \label{claim_1}
The $-1$ eigenspace of $r_{N+1}$ acting on $\ker L_u$ is trivial.
\end{claim}

\begin{proof}[Proof of \cref{claim_1}]
Suppose we have $ \phi \in \ker L_u$ such that 
$$
\phi \circ r_{N+1}= - \phi.
$$
This implies that $\phi = 0$ along $\partial S^{N}_+$, where $S^{N}_+=\{x \in S^{N} \ : \ x_{N+1}>0\}$.

Now, we can multiply the equation \eqref{eq1} by $\phi$ on $S^{N}_+$ and integrate by parts to obtain
$$
\e^2 \int_{S^{N}_+} \nabla u \nabla \phi
=
- \e^2 \int_{S^{N}_+} \phi \Delta u 
 = 
 - \int_{S^{N}_+} f(u) \phi,
$$
since the boundary term vanishes as $\phi = 0$ along $\partial S^{N}_+$.
Similarly, we can use the equation for $\phi$ to find 
$$
\e^2 \int_{S^{N}_+} \nabla u \nabla \phi 
= 
- \e^2 \int_{S^{N}_+} u \Delta \phi
=
- \int_{S^{N}_+} f'(u) u \phi,
$$
using that $u=0$ along $\partial S^{N}_+$.
Combining both of these we arrive at
\begin{equation} \label{eq_null_1}
\int_{S^{N}_+} (f'(u) u - f(u)) \phi =0
\end{equation}
By our normalization of $u$ we have that $u > 0$ in $S^{N}_+$, hence the second  assumption on $f$ 
implies that also 
\begin{equation} \label{eq_null_2}
f'(u) u - f(u) > 0 \ \text{in} \ S^{N}_+.
\end{equation}
On the other hand, $u$ is assumed to have index $1$, we have that $\phi$ is a second eigenfunction of $L_u$.
By Courant's nodal domain theorem this implies that $S^{N} \setminus \{ \phi =0\}$ has precisely
two connected components.
As remarked above we also know that $\phi =0$ along $\partial S^{N}_+$.
Combining the last two pieces of information we find that 
\begin{equation} \label{eq_null_3}
\phi \geq 0 \ \text{in} \ S^{N}_+
\end{equation}
up to multiplying $\phi$ by $-1$.
Combining \eqref{eq_null_1},\eqref{eq_null_2}, and \eqref{eq_null_3}, we find that $\phi=0$, which is precisely our claim.
\end{proof}

We now prove a similar assertion for the reflections $r_v$ with $v \in S^{N-1}$.

\begin{claim} \label{claim_2}
Let $v \in S^{N-1}$, then the the $-1$ eigenspace of $r_v$ acting on $\ker L_u$ is one-dimensional.
\end{claim}

\begin{proof}[Proof of \cref{claim_2}]
Suppose that $0\neq \phi \in \ker L_u$ has $\phi \circ r_v = - \phi$ for some $v \in S^{N-1}$.
Then (up to multiplying $\phi$ by $-1$) it follows as in the proof of the first claim above by the Courant nodal domain theorem that 
$$
\phi > 0 \ \text{in} \ \{x \in S^{N} \ : \langle x , v \rangle > 0\}
$$ 
and  
$$
\phi = 0 \ \text{along} \ \langle v \rangle^\perp \cap S^{N}.
$$
This implies that $\phi$ is a \emph{first} Dirichlet eigenfunction for $L_u$ on the hemisphere $\{x \in S^{N} \ : \langle x , v \rangle \geq 0\}$.
Let us now choose some $w \in \langle v \rangle^\perp \cap \langle e_{N+1} \rangle^\perp \cap S^{N}$ and let $X$ be a non-trivial Killing field generated by a rotation fixing $w$, i.e.\ $X(w)=0$, then we have the corresponding function $\phi_X \in \ker L_u$ described earlier
\footnote{Note that different choices of the Killing field $X$ result in the same function $\phi_X$ up to scaling.}.
Moreover, we have that
$$
\phi_X > 0 \ \text{in} \ \{x \in S^{N} \ : \langle x , v \rangle > 0\}
$$ 
and  
$$
\phi_x = 0 \ \text{along} \ \langle v \rangle^\perp \cap S^{N}.
$$
In particular, it follows that $\phi_X$ is first Dirichlet eigenfunction as well.
This implies that $\phi \in \langle \phi_X \rangle$, which is what we claimed.
\end{proof}

\smallskip

Let us now finish the proof using the above two claims.
We denote by 
$$
V=\{ \phi_X \ : \ X \ \text{Killing field} \} \subset \ker L_u
$$
 the $n$-dimensional subspace spanned by the eigenfunctions generated by rotations. 
We want to show that the above inclusion is an equality.
Let $\phi \in V^\perp \subset \ker L_u$ and assume that $\phi \neq 0$.

Note that for $v \in S^{N-1} \cup \{e_{N+1}\}$  the reflections $r_v$ preserve $V$ and act by isometries on $\ker L_u$ (endowed with the $L^2$ scalar product).
In particular, each of these defines an involution on $V^\perp \subset \ker L_u$.
It follows from \cref{claim_1} and \cref{claim_2} that the $-1$ eigenspace of any of these has to be trivial.
This implies that $\phi \in V^\perp$ is symmetric about the $e_{N+1}$-axis and invariant under the reflection at $\{x_{N+1}=0\}$.
In particular 
$$
\phi = c \ \text{along} \  S^{N-1}
$$
for some constant $c \in \mathbb{R}$.
We can not have $c=0$,
since otherwise this would imply by the Courant nodal domain theorem that $\phi$ has a sign in the 
upper and lower hemisphere contradicting our computation from \cref{claim_1} (and also the Hopf boundary Lemma).
But then, since $\phi$ is not a first eigenfunction and invariant under $r_{N+1}$ there has to be some point $z=(z',z_{N+1}) \in S^{N}_+$ with $z_{N+1}>0$.
with $\phi(z)=0$.
Since $\phi$ is rotationally symmetric about the $e_{N+1}$ axis, this implies that the set 
$$
\{(z',z_{N+1}) \in S^{N} \} \subseteq \{\phi = 0\}. 
$$
Since $\phi$ is invariant under $r_{n+1}$ this implies that also
$$
\{(z',-z_{N+1}) \in S^{N} \} \subseteq \{\phi = 0\}. 
$$
But this implies that $S^{N} \setminus \{ \phi =0\}$ has at least three connected components contradicting Courant's nodal domain theorem if $\phi \neq 0$.
\end{proof}

\subsection{The energy gap}
Thanks to \cref{prop_nullity} we can now provide the argument for \cref{gap}.
We argue by contradiction, essentially exploiting the fact that we have shown that our discussion up to this point could be summarized as the Allen--Cahn functional being Morse-Bott near the energy level $a_\e$.
Alternatively, one could invoke an appropriate version of the implicit function theorem.

\begin{proof}[Proof of \cref{gap}]
Assume that we have a sequence $(v_{j})_{j \in \mathbb{N}}$ of solutions to \eqref{AC} with
$$
E_\e(v_{j}) > a_\e \ \text{and} \ \lim_{j \to \infty} E_\e(v_{j}) = a_\e
$$
Note that \cref{mainac} 
can be stated as
$$
A_\e =\{ u \in W^{1,2}(S^{N}) \ : \ u \ \text{solves} \ \eqref{AC} \ \text{and} \ E_\e(u)=a_\e \}
=
\{u_{0}\ \circ \ R \ : \ R \in O(N+1) \},
$$
where $u_0$ is rotationally symmetric, odd and monotone in the radial direction. In particular, the elements $A_\e$ are determined by where their maximum is, so $A_\e$ is a sphere of dimension $N$. In addition, since $A_\e$ is compact we have:
\begin{equation} \label{eq_footpt}
\inf_{R \in O(N+1)}\| v_{j} - (u_{0} \circ R) \|_{L^2}
=
\| v_{j} - (u_{0} \circ R_j) \|_{L^2}
\end{equation}
for some $R_j \in O(N+1)$.
Notice that since $u_0\circ R_j$ minimizes the distance from $v_j$ to $A_\e$, it follows that $v_j-u_0$ is orthogonal to the tangent space of $A_\e$. By composing everything with $R_j^{-1}$ we may assume that $R_j=\id$ for any $j \in \mathbb{N}$.  

By standard elliptic estimates, we have that 
$v_{j} \to v \ \text{in} \ C^{\infty} $ and $E_\e(v) = a_\e.$
It follows that 
$v=u_{0}.$
Consider now the sequence of functions
$w_{j} = \alpha_{j} (v_{j}-u_{0}),$
with $\alpha_{j}^{-1} = \| v_{j}-u_{0} \|_{L^2}$, so that $\|w_{j}\|_{L^2}=1$.
It follows from standard arguments
that 
$$
w_{j} \to \phi \in \ker L_{u_{0}},
$$
where the convergence is smoothly and hence $\|\phi\|_{L^2}=1$.
But by the choice \eqref{eq_footpt} we have that $\phi \perp \ker L_{u_{0}}$ thanks to \cref{prop_nullity}, which is a contradiction.
\end{proof}

Thanks to the Palais--Smale condition satisfied by the Allen--Cahn functional the proof of the gap for the Allen--Cahn widths does not rely on \cref{gap} contrary to the argument leading to the same assertion for the Almgren--Pitts widths.

\begin{proof}[Proof of \cref{thm_width_gap}]
Suppose that
$$
c_\e(1) = \dots = c_{\e}(N+2)
$$
for the Allen--Cahn widths on the sphere $S^N$.
Under this assumption it follows from \cite[Theorem 3.3 (2)]{ACClosed} that the cohomological $\mathbb{Z}/2$ index of the set
$$
K_{c_{\e}(1)} = \{ u \in W^{1,2}(S^N) \ : \ E_\e'(u)=0, E_{\e}(u)=c_{\e}(1) \}
$$
satisfies
\begin{equation} \label{eq_coh_ind}
\Ind_{\mathbb{Z}/2}(K_{c_\e(1)}) \geq N+2.
\end{equation}
On the other hand, we know that there is a $\mathbb{Z}/2$-equivariant map
$$
\phi \colon K_{c_\e(1)} \to S^N.
$$
Since $H^{nN+1}(S^N;\mathbb{Z}/2)=0$ this implies that
$$
\Ind_{\mathbb{Z}/2}(K_{c_\e(1)}) \leq N+1
$$
contradicting \eqref{eq_coh_ind}.
\end{proof}

\section{Bifurcation at the first positive critical level}\label{sec:bif}

In this section, we will study the bifurcation for solutions of \eqref{AC} on $S^3$ which occurs at $\e = \e_1 = (\lambda_1(S_+^3))^{-1/2}$. We recall that the only solutions of \eqref{AC} for any $\e \geq \e_1$ are the constants $\pm 1$ and $0$. Denote by $A_\e \subset W^{1,2}(S^3)$ the set of all unstable solutions of least energy in $S^3$. By Theorem \ref{mainac}, the set $A_\e$ is diffeomorphic to $S^3$. Our goal is to prove \cref{thm:bifurcation}.

We begin by constructing the families of nonradially symmetric solutions mentioned in (2). We regard $S^3$ as the set of all $(z,w) \in \C^2$ such that $|z|^2+|w|^2=1$. Let $\mathcal{T} \subset S^3$ be the Clifford torus, namely
	\[\mathcal{T} = \{ x \in S^3 : x_1^2+x_2^2 = 1/2 = x_3^2 + x_ 4^2\}.\]
It is a minimal surface and it bounds the solid torus
	\begin{equation} \label{handle}
		\Omega_{\mathcal{T}} = \{ x \in S^3 : x_1^2 + x_2^2 < 1/2\}.
	\end{equation}
Moreover, $\mathcal{T}$ is the nodal set of the restriction of the harmonic polynomial $p(x) = x_1^2 + x_2^2 - x_3^2 - x_4^2$ on $\R^4$ to the sphere. Since this restriction is an eigenfunction for $\Delta=\Delta_{S^3}$ with associated eigenvalue $\lambda_2(S^3)=8$, we get $\lambda_1(\Omega_{\mathcal{T}}) = \lambda_2(S^3)=8$. Finally, we note that $\mathcal{T}$ is invariant by the isometry
	\[s:(x_1,x_2,x_3,x_4) \in S^3 \mapsto (x_3,x_4,x_1,x_2) \in S^3,\]
which switches $\Omega_{\mathcal{T}}$ and the interior of its complement, and that $\Omega_{\mathcal{T}}$ is invariant by the isometries
	\[f_{\theta,\rho}: (z,w) \in S^3 \subset \C^2 \mapsto (e^{i\theta}z,e^{i\rho}w) \in S^3,\]
for all $\theta,\rho \in \R$. By the uniqueness of positive Dirichlet solutions, we conclude:

\begin{proposition} \label{clifford}
For any $\e \in (0,\e_2)$, there is a solution $u$ of \eqref{AC} whose nodal set is precisely the Clifford torus. Moreover, it is invariant by $f_{\theta,\rho}$, for all $\theta,\rho \in \R$, and it satisfies
	\[u(w,z) = -u(z,w), \quad \mbox{for all} \quad (z,w) \in S^3.\]
\end{proposition}

\begin{proof}
By Theorem \ref{Dirichlet}, there is a unique positive solution $v \in C^3(\Omega_{\mathcal{T}})$ of \eqref{AC} in $\Omega_{\mathcal{T}}$ which vanishes on $\partial \Omega_{\mathcal{T}} = \mathcal{T}$, provided $\e < \e_2$. Since $v \circ f_{\theta,\rho}$ also solves \eqref{AC} on $\Omega_{\mathcal{T}}$, we get $v \circ f_{\theta,\rho} = v$, for all $\theta,\rho \in \R$. Moreover,
	\[\SO(2)\times\SO(2) = \{f_{\theta,\rho}: (z,w) \mapsto (e^{i\theta}z,e^{i\rho}w) \, : \, (\theta,\rho) \in \R^2\}\]
acts transitively on $\mathcal{T}$, hence the normal derivative of $u$ is constant along the boundary $\mathcal{T}$.

Therefore, the solution $u:S^3 \to \R$ with the desired properties is given by
	\[u(x) = \left\{ \begin{array}{rl} v(x), & \mbox{if} \ x \in \bar\Omega_{\mathcal{T}}, \\ -v(s(x)), &\mbox{if}\ x \in S^3 \setminus \bar\Omega_{\mathcal{T}} \end{array} \right. .\qedhere\]
\end{proof}

Similarly, the set
	\[\mathcal{X} = \{x \in S^3 : x_3\cdot x_4 = 0\},\]
is the union of two orthogonal equators, and the nodal set of the restriction of the harmonic polynomial $x \mapsto x_3\cdot x_4$ to $S^3$. Since this polynomial is also a Laplace eigenfunction associated to $\lambda_2(S^3)$ and it is positive in the region $\Omega_{\mathcal{X}} = \{x \in S^3 : x_3>0, x_4>0\}$, we see that $\lambda_1(\Omega_{\mathcal{X}}) = \lambda_2(S^3)$. By Theorem \ref{Dirichlet}, there is a unique positive Dirichlet solution $u_{\mathcal{X}}$ of \eqref{AC} in $\Omega_{\mathcal{X}}$.\smallskip

Observe that $\Omega_{\mathcal{X}}$ is invariant by the isometry $t(x_1,x_2,x_3,x_4) = (x_1,x_2,x_4,x_3)$, which interchanges the two orthogonal equators in $\mathcal{X}$. It is also invariant by any $T \in \mathrm{O}(2)$ acting on the first two coordinates of $x \in S^3$. Hence $u_{\mathcal{X}} \circ t = u_{\mathcal{X}}$, and $u_{\mathcal{X}}$ depends on $x_3$ and $x_4$ only. Therefore, we can extend $u_{\mathcal{X}}$ to a solution $\bar u$ in $S^3$ by odd reflections across $S^3$, namely
	\[\bar u(x_1,x_2,x_3,x_4) = \sgn(x_3x_4) \cdot u_{\mathcal{X}}(x_1,x_2,\sgn(x_3)x_3,\sgn(x_4)x_4).\]
This concludes the construction of the second family of nonradially symmetric solutions for $\e<\e_2$, and finishes the proof of Theorem \ref{thm:bifurcation} (2).

In order to prove Theorem \ref{thm:bifurcation} (1), we will first rule out other radially symmetric solutions. More precisely,

\begin{lemma}\label{radsy} If $u$ is a nonconstant radially symmetric solutions of \eqref{AC} for $\varepsilon\in [\varepsilon_{2},\varepsilon_{1})$, then  $u \in A_{\varepsilon}$.
\end{lemma}

Before proving the lemma above, we recall some facts about the first Dirichlet eigenvalue of certain domains in $S^3$. For any geodesic ball $B_\tau \subset S^3$, where $\tau \in (0,\pi)$, we have (see \cite{Bang})
	\[\lambda_1(B_\tau) = \left(\frac{\pi}{\tau}\right)^2 - 1.\]
In particular, $\lambda_1(B_\tau) \geq \lambda_2(S^3) = 8$ if, and only if, $\tau \leq \pi/3$. Consider also the spherical segment
	\[\Omega_h = \{ x \in S^3 : \dist(x,\{x_4=0\})< h\}= B_{\pi/2+h}(e_4) \setminus \bar B_{\pi/2-h}(e_4).\]
for $h \in (0,\pi/2)$. We claim that
	\[\lambda_1(\Omega_h) \geq \left(\frac{\pi}{2h}\right)^2 -1.\]
In fact, if $\phi$ is a positive eigenfunction corresponding to $\lambda_1(\Omega_h)$, then we may write $\phi(x) = f(r(x))$ for some function $f \in C^\infty([\pi/2-h,\pi/2+h])$ which vanishes on the boundary of its domain, where $r(x) = \dist(x,e_4)$. Then
	\[\lambda_1(\Omega_h) = \frac{\int_{\Omega_h}|\nabla \phi|^2}{\int_\Omega |\phi|^2} = \frac{\int_{\pi/2-h}^{\pi/2+h} \sin^2(t)f'(t)^2\,dt}{\int_{\pi/2-h}^{\pi/2+h} \sin^2(t)f(t)^2\,dt}.\]
Using
	\[{\int_{\pi/2-h}^{\pi/2+h} \left(\frac{d}{dt}(\sin(t)f(t))\right)^2\,dt} = {\int_{\pi/2-h}^{\pi/2+h} \sin^2(t)f'(t)^2\,dt} + {\int_{\pi/2-h}^{\pi/2+h} \sin^2(t)f(t)^2\,dt},\]
which follows by integration by parts and $f(\pi/2-h)=0=f(\pi/2+h)$, and Wirtinger's inequality, we obtain
	\[\lambda_1(\Omega_h) = \frac{{\int_{\pi/2-h}^{\pi/2+h} \left(\frac{d}{dt}(\sin(t)f(t))\right)^2\,dt}}{\int_{\pi/2-h}^{\pi/2+h} \left(\sin(t)f(t)\right)^2\,dt}-1 \geq \left(\frac{\pi}{2h}\right)^2 -1 .\]

\begin{proof}[Proof of Lemma \ref{radsy}]
By composing $u$ with an isometry, we may assume that $u$ is radially symmetric with respect to $e_4 \in S^3$. The nodal set $u^{-1}(0)$ is the union of concentric geodesic spheres centered at $e_4$, and, by the maximum principle, $u(\pm e_4) \neq 0$. Hence, the connected component of $\{u^{2}> 0\}$ containing $e_4$ is a geodesic ball $B_\tau=B_\tau(e_4)$. By Theorem \ref{Dirichlet}, we obtain $\varepsilon_{2} \leq\varepsilon < \lambda_1(B_\tau)^{-1/2}$, so $\lambda_1(B_\tau) < \lambda_2(S^3)$ and $\tau > \pi/3$. Similarly, the connected component of $\{u^2>0\}$ containing $-e_4$ is a geodesic ball of radius $>\pi/3$.

We claim that $u^{-1}(0)$ is connected. In fact, if this is not the case, then $\{u^2>0\}$ has a third connected component $\Omega$ which is contained in the spherical segment $\Omega_{\pi/6}$. Consequently,
	\[\lambda_1(\Omega)>\lambda_1(\Omega_{\pi/6}) \geq \frac{\pi^2}{(\pi/3)^2} -1 = 8 = \lambda_2(S^3).\]
But, using Theorem \ref{Dirichlet} again, we get $\varepsilon_2 \leq \varepsilon < \lambda_1(\Omega)^{-1/2} \leq \lambda_2(S^3)^{-1/2}$, which is a contradiction. This shows that $u^{-1}(0)$ is a single geodesic sphere. By Lemma \ref{nodal-equator}, we see that $u^{-1}(0) = \partial B_{\pi/2}=\{x_4=0\}$. Hence, the uniqueness part of Theorem \ref{Dirichlet} yields $u \in A_{\varepsilon}$.
\end{proof}

\begin{remark}
For sufficiently small $\e>0$, there are radially symmetric solutions of the Allen-Cahn equation on $S^3$ which are not in $A_\e$; see Example 1 in \cite{gmn_ac_19}. The previous lemma shows that $\e<\e_2$ whenever such solutions exists.
\end{remark}

We recall some facts about Lie group actions.  Let $G$ be a compact Lie group which acts differentiably on the right on a (Hilbert) manifold $\mathcal{M}$. For any $u \in \mathcal{M}$, denote by
		\[G_u = \{ g \in G : u \cdot g = u\} \quad \mbox{and} \quad u \cdot G = \{u \cdot g: g \in G\}\]
	the \emph{isotropy group} of $u$ and the \emph{orbit} of $u$, respectively. Then $G_u$ is a closed Lie subgroup of $G$, and the quotient manifold $G/G_u$ can be embedded into $\mathcal{M}$, with image $u\cdot G$. In particular, $G/G_u$ is diffeomorphic to $u \cdot G$ and
		\begin{equation} \label{DimensionsLie}
			\dim (u \cdot G) = \dim G - \dim G_u,
		\end{equation}
	see e.g. \cite[\S VI.1]{Bredon} for a proof. 
	
	The group of all orientation preserving isometries of $S^3$ is the special orthogonal group $\SO(4)$. Since any $T \in \SO(4)$ is a diffeomorphism of $S^3$, this group acts on the Sobolev space $W^{1,2}(S^3)$ on the right by composition, i.e. $u \cdot T = u\circ T$. This yields a differentiable action $W^{1,2}(S^{N}) \times \SO(4) \to W^{1,2}(S^{N})$. Moreover, $u\cdot T$ is a solution of \eqref{AC} whenever $u$ is a solution, and the orbit of any $u\in A_{\varepsilon}$ under this action is precisely the set $A_{\varepsilon}$.
	
We will need the following classification of Lie subgroups of $\SO(4)$.

\begin{theorem}[\cite{Wakakuwa}] \label{classificationSO4}
Let $G$ be a connected Lie subgroup of $\SO(4)$ of dimension $\geq 2$. Up to conjugation in $\SO(4)$, $G$ is one of the following subgroups:
\begin{enumerate}
	\item $\SO(4)$, if $\dim G = 6$;
	\item $\mathrm{U}(2)$ (unitary complex $2\times 2$ matrices acting on each $(z,w) \in S^3$ by matrix multiplication), if $\dim G=4$.
	\item $\mathrm{SU}(2)$ (unitary complex $2\times 2$ matrices with $\det =1$ acting on each $(z,w) \in S^3$ by matrix multiplication), or $\SO(3) = {\{T \in \SO(4) : Te_4=e_4\}}$, where $e_4=(0,i) \in S^3$, if $\dim G =3$.
	\item The torus $\SO(2)\times \SO(2)=\{f_{\theta,\rho} : \theta,\rho \in \R\}$, if $\dim G=2$.
\end{enumerate}
In particular $\dim G \neq 5$.
\end{theorem}

\begin{lemma} \label{IsotropyandSymmetry}
Let $\e \in (\e_2,\e_1)$, and let $u$ be a nonconstant solution of \eqref{AC}. If the group
	\[G=\SO(4)_{u} = \{T \in \SO(4) : u \circ T = u\}\]
has dimension $\geq 2$, then $u \in A_\e$. 
\end{lemma}

\begin{proof}
By Theorem \ref{classificationSO4}, we can't have $\dim G \geq 4$. In fact, after possibly replacing $u$ with $u \circ T$ for some isometry $T$ (and consequently $G_{u}$ with $T^{-1}G_{u}T$), $G$ is either $\SO(4)$ or $\mathrm{U}(2)$. Both groups act transitively on $S^3$, so $u$ would be constant. Hence, it suffices to consider the following cases:\\

\noindent \textit{Case 1:} $\dim G = 3$.

\noindent Again, after possibly composing $u$ with an isometry, we may assume $G$ is either $\mathrm{SU}(2)$ or $\SO(3)$. The former case cannot happen, as $\mathrm{SU}(2)$ also acts transitively on $S^{3}$. In the latter case, $G$ acts by rotations in each hyperplane orthogonal to $(0,0,0,1)\in S^{3}$, so $u$ is rotationally symmetric with respect to this point. By Lemma \ref{radsy} we get $u \in A_{\varepsilon}$.\\

\noindent \textit{Case 2:} $\dim G = 2$.

\noindent We will show that this case doesn't happen. By Theorem \ref{classificationSO4}, we may assume that $G$ is the group $\SO(2)\times \SO(2)$. The action of $G$ on $S^{3}$ has two orbit types: 2-dimensional tori given by the boundary of the region 
\[\Omega_r = \{(z,w) \in S^{3}: |z| \leq r\}, \quad \mbox{for some}  \quad  r\in (0,1),\]
and the circles $\{|z|=0\} \cap S^3$ and $\{|z|=1\}\cap S^3$. If $u \not\equiv \pm 1$, then $u$ changes sign and we may pick $x \in u^{-1}(0)$. By the maximum principle and $x \cdot G \subset u^ {-1}(0)$, the orbit $x\cdot G$ cannot be a circle. It follows that $u$ has a nodal domain $\Omega$ which is contained in either $\Omega_{\mathcal{T}}$ (see \eqref{handle}) or in its complement. In any case, we see that $\lambda_1(\Omega) \leq \lambda_1(\Omega_{\mathcal{T}}) = \lambda_2(S^3)$, so Theorem \ref{Dirichlet} implies $\e< \lambda_2(S^3)^{-1/2} = \e_2$, contradicting our assumption on $\e$.
\end{proof}

\begin{lemma}\label{LargeIsotropy}
Given any solution $u$ of \eqref{AC} for $\e \in (\e_2,\e_1)$, the group ${G=\SO(4)_u}$ has dimension $\geq 2$.
\end{lemma}

\begin{proof}
Since $W''(u) \geq -1 = W''(0)$, we have
	\[\int_{S^3} \e|\nabla \phi|^2 + \frac{W''(u)}{\e} \phi^2 \geq \int_{S^3} \e|\nabla \phi|^2 + \frac{W''(0)}{\e}\phi^2, \quad \mbox{for any} \quad \phi \in W^{1,2}(S^3).\]
If $L_u$ and $L_0$ are the linearizations of the Allen-Cahn operator at $u$ and $0$ respectively, the inequality above yields
	\[\lambda_k(L_u) \geq \lambda_k(L_0), \quad \mbox{for all} \quad k \in \N.\]
Since $L_0 = -\e^2\Delta + W''(0)= -\e^2 \Delta-1$, we see that $\lambda_k(L_0) = \e^2\lambda_k(S^3) -1$. Using that $\lambda_1(S^3)$ has multiplicity $4$ and $\e \in (\e_2,\e_1)$, we obtain
	\[\Ind(u) + \dim\ker(L_u) \leq \Ind(0) + \dim\ker(L_0) = 5.\]
By \cite{FSV}, we have $\Ind(u) \geq 1$, so $\dim\ker(L_u) \leq 4$. Since the tangent vector of any curve through $u$ in the orbit $u\cdot \SO(4)$ is an element in $\ker(L_u)$, we see that
	\[\dim(u\cdot \SO(4)) = \dim T_u(u\cdot \SO(4)) \leq \dim\ker(L_u) \leq 4.\]
Along with \eqref{DimensionsLie} and $\dim \SO(4) = 6$, this shows $\dim(\SO(4)_u) \geq 2$.
\end{proof}

\begin{remark}
The same argument gives a lower bound for the dimension isotropy group of a nonconstant solution $u$ of \eqref{AC} in $S^{N}$, for $\e$ in the range $(\lambda_2(S^{N})^{-1/2},\lambda_1(S^{N})^{-1/2})$. In fact, since $\lambda_1(S^{N})$ has multiplicity $n+1$ and $\dim \SO(N+1) = \frac{N(N+1)}{2}$, for any such $\e$ and $u$, the group $\SO(N+1)_u$ has dimension $\geq \frac{(N-2)(N+1)}{2}=\dim(\SO(N))-1$.
\end{remark}

As a consequence of Lemmas \ref{IsotropyandSymmetry} and \ref{LargeIsotropy}, we see that the only solutions of \eqref{AC} with $\e \in (\e_2,\e_1)$ are the constants $\pm 1$ and $0$, and the least positive energy solutions $A_\e$. This finishes the proof of Theorem \ref{thm:bifurcation}.

\bibliography{least-energy}
\bibliographystyle{acm}

\end{document}